\documentclass[12pt,longbibliography]{article}
\usepackage[utf8]{inputenc}
\usepackage[a4paper, margin=2.7cm]{geometry}
\usepackage{amsmath, amsthm, amsfonts, amssymb}
\usepackage{mathrsfs}
 \usepackage[colorlinks=true,linkcolor=blue,citecolor=blue,urlcolor=blue,breaklinks]{hyperref}
 \usepackage{hyperref}
 \usepackage{amsmath}
\newtheorem{theorem}{Theorem}[section]
\newtheorem{lemma}{Lemma}[section]
\newtheorem{remark}{Remark}[section]
\usepackage{indentfirst}
\usepackage{enumerate}

\numberwithin{equation}{section}

\usepackage{color}
\def\re{\color{red}}

\newcommand{\titre}[1]{\begin{center}
		{\Large{\bf #1}}
\end{center}}

\begin{document}
	
\titre{Ground states of a coupled pseudo-relativistic Hartree system: existence and concentration behavior}
	\begin{center}
Huiting He$^{*}$, Chungen Liu$^{*}$, Jiabin Zuo$^{*}$
\footnote{Email addresses: Heht21@163.com(Huiting He), liucg@nankai.edu.cn(Chungen Liu), zuojiabin88@163.com(Jiabin Zuo)\\
 $^{*}$ School of Mathematics and Information Science, Guangzhou University, Guangzhou, 510006, P. R. China.}

		\end{center}

	\vspace*{+0.2cm}
	
	
	\begin{abstract}
		This paper is concerned with the ground states of a coupled pseudo-relativistic Hartree system in $\mathbb{R} ^{3} $ with trapping potentials, where the intraspecies and the interspecies interaction are both attractive. By investigating an associated constraint minimization problem, the existence and non-existence of ground states are classified completely. Under certain conditions on the trapping potentials, we present a precise analysis on the concentration behavior of the minimizers as the coupling coefficient goes to a critical value, where the minimizers blow up and the maximum point sequence concentrates at a global minima of the associated trapping potentials. We also identify an optimal blowing up rate under polynomial potentials by establishing some delicate estimates of energy functionals.
	\end{abstract}
	\vspace{0.6cm}

\section{Introduction}
\quad \ In this paper, we are interested in the following a coupled pseudo-relativistic Hartree system with trapping potentials in  $\mathbb{R} ^{3}$
\begin{equation}\label{eP1}
\begin{cases}
\sqrt{- \triangle +  m^{2} }u_{1}+V_{1}(x)_{}u_{1}=\mu u_{1}+a_{1}(\left | x \right | ^{-1}\ast u_{1} ^{2}  )u_{1}+\beta (\left | x \right | ^{-1}\ast u_{2} ^{2} )u_{1},\\
\sqrt{- \triangle +  m^{2} }u_{2}+V_{2}(x)_{}u_{2}=\mu u_{2} +a_{2}(\left | x \right | ^{-1}\ast u_{2} ^{2}  )u_{2}+\beta (\left | x \right | ^{-1}\ast u_{1} ^{2} )u_{2},
\end{cases}
\end{equation}
where $\mu \in \mathbb{R}$ is the suitable Lagrange multiplier, the convolution kernel $\left | x \right | ^{-1}$ denotes the Newtonian gravitational potential, $a_{i}> 0~(i=1,2)$, $\beta > 0$ represent that both the intraspecies interaction inside each component and the interspecies interaction between two components are attractive.

As shown by Lieb and Yau in their seminal work \cite{J22}, system (\ref{eP1}) arises from a quantum mechanical model describing the stellar collapse of the neutron stars. Moreover, in celebrated paper  \cite{Long21}, it was found that such systems can be used to give an effective description on the dynamics of $N$-body bosons with relativistic dispersion  interaction in the mean-field limit. Another physical application of system (\ref{eP1}) was presented in \cite{Amrouss23}, where the nonlinear time-dependent evolution equation is proposed by
\begin{equation}\label{e300}
i\partial_{ t} \psi=\sqrt{- \triangle +  m^{2} }\psi+V(x)\psi-a\left (  \left | x \right | ^{-1}\ast \psi  ^{2}   \right ) \psi \quad {\rm in} \ \mathbb{R} ^{3}.
\end{equation}	
Here, $\psi (t,x) $ is conceived to be a complex-valued wave field and the travelling solitary waves of equation (\ref{e300}) are given by $\psi (t,x)=e^{-i\mu t}u(x) $, where $u(x)$ is a ground state of the associated nonlinear Hartree equation
\begin{equation}\label{e301}
 \sqrt{- \triangle +  m^{2} }u+V(x)u=\mu u+a(\left | x \right | ^{-1}\ast u ^{2}  )u \quad {\rm in} \ \mathbb{R} ^{3}.
 \end{equation}	
 For more physical details, we refer the readers to references \cite{ousbika3,ousbika4,J30,J31}. Also, one can see \cite{ousbika,Ricc,Zhou21} for some interesting results on the classical nonlinear Choquard equation, i.e., the external potential $V(x)=0$ in (\ref{e301}), which occurs in a one-component plasma as an approximation to the Hartree-Fock theory.

 When the non-local operator $\sqrt{- \triangle +  m^{2} }$ is replaced by the Laplacian $-\triangle$ and the Hartree nonlinearity is replaced by pure power type nonlinearity in (\ref{e301}), that is the case of  well-known Gross-Pitaevskii system arising in studying the Bose-Einstein condensate phenomenon with external potential, which has attracted a great deal of attention over the past few years.
 In particular, it is worth mentioning the very relevant contributions for one-component equations by Guo and his collaborators in their series work \cite{Guo24,Guo28,Guo25}, where the authors identified the existence intervals of the ground states and presented a detailed description of the collapse behavior under different types of external potentials.

 Note that, in the muti-component framework,  one has to take into account a certain interspecies interaction between each components, which may display more complicated physical phenomena and thus makes the analysis more challenging. Stimulated by the recent work \cite{Guo26} and \cite{Guo27} , where the authors investigated the ground states of a coupled Gross-Pitaevskii system with attractive interspecies interaction, the existence results and the concentration behavior of the ground states concerning the Laplacian systems with Hartree nonlinear terms were also analyzed in references \cite{Du22} and \cite{J33}, tackling the difficulties arising from the Hartree nonlinearities. Moreover, one can see \cite{Zhou22} for the case with repulsive interspecies interaction.

 For the main topic in the present paper, that is, a non-local system with both pseudo-relativistic operator and nonlinear Hartree terms, some interesting results arise comparing to the classical local case. As recently shown by Yang and Yang in \cite{Bona16}, where a one-component non-local system was considered, they established a delicate energy estimate on the pseudo-relativistic term by using the Fourier transform. Another similar result is proposed by Nguyen in \cite{Xiong21}, where the analysis of the concentration behavior related to the one-component non-local system was constructed under three different types of external potentials. Also, one can see the recent study on the concentration behavior of the minimizers as the stellar mass tends to a certain critical value in \cite{Zheng19}. In addition, it is worth noticing that Wang, Zeng and Zhou very recently generalized the one-component case to the coupled pseudo-relativistic Hartree system (\ref{eP1}) in \cite{Hkhani15}, where the existence results involving certain ranges for the three parameters $a_{1},a_{2},\beta$ were presented by establishing an appropriate auxiliary functional.

 However, to the authors' knowledge, the concentration behavior of the minimizers has rarely been discussed concerning the coupled pseudo-relativistic Hartree system (\ref{eP1}). Inspired by the above work, main results of this paper are focused on not only the existence range of solutions but also more importantly the concentration behavior of the minimizers of {\re a coupled} pseudo-relativistic Hartree system with trapping potentials $V_{i}(x)$ satisfying for $i=1,2$,
 \begin{enumerate}[($\Omega$)]
\item  $0\le V_{i}(x)\in L_{loc}^{\infty }(\mathbb{R} ^{3} ), \lim_{\left | x \right |  \to \infty}V_{i}(x)=\infty, \inf_{x\in \mathbb{R} ^{3} }(V_{1}(x) +V_{2}(x))$ is attained and $\inf_{x\in \mathbb{R} ^{3} }(V_{1}(x) +V_{2}(x))=0$.
\end{enumerate}

To begin with, we now introduce some notations and recall some known results. Let the Sobolev space
\begin{equation}\label{e23}
H^{\frac{1}{2} }(\mathbb{R} ^{3} )=\left \{ u\in L^{2}(\mathbb{R} ^{3};\mathbb{R}):\left \| u \right \|_{H^{\frac{1}{2} }(\mathbb{R} ^{3} )}^{2}:=\int\limits_{\mathbb{R} ^{3}}\left ( 1+2\pi \left | s  \right |    \right )\left | \mathcal{F}u\left ( s   \right ) \right |  ^{2} ds  < \infty        \right \},
\end{equation}	
where $\mathcal{F}u$ stands for the Fourier transform of $u$.
The space $H^{\frac{1}{2} }(\mathbb{R} ^{3} )$ endowed with the inner product
\begin{equation}\label{e24}
\left (u,v\right )_{H^{\frac{1}{2} }(\mathbb{R} ^{3} ) }=\int\limits_{\mathbb{R} ^{3}}\left ( 1+2\pi \left | {  s}  \right |  \right )\overline{\mathcal{F}u({  s})}   \mathcal{F}v\left ( {  s}  \right )d{  s}
\end{equation}
is easily seen to be a Hilbert space. The pseudo-relativistic operator $\sqrt{- \triangle +  m^{2} }$ is defined in Fourier space as multiplication by $\sqrt{\left |2\pi {  s}  \right |  ^{2}+m^{2}}$, i.e.,
\begin{equation}\label{e25}
\sqrt{- \triangle +  m^{2} }u=\mathcal{F} ^{-1}\left ( \sqrt{\left |2\pi {  s}  \right |  ^{2}+m^{2}}\mathcal{F}u({  s} )   \right ), \quad  u\in H^{\frac{1}{2} }(\mathbb{R} ^{3}).
\end{equation}
Moreover, the norm $\left \| \cdot  \right \|_{H^{\frac{1}{2} }(\mathbb{R} ^{3} )}$ can also be given by
\begin{equation}\label{e26}
\left \| u  \right \|_{H^{\frac{1}{2} }(\mathbb{R} ^{3} )}^{2}=\displaystyle\int\limits_{\mathbb{R} ^{3}}\left [ u\left ( -\triangle  \right ) ^{\frac{1}{2} }u+u^{2}   \right ]dx=\displaystyle\int\limits_{\mathbb{R} ^{3}}\left [ \left | \left ( -\triangle  \right ) ^{\frac{1}{4} }u \right | ^{2}+u^{2} \right ] dx,
\end{equation}
one can refer \cite{Galewski} for more details. For ground states of system (\ref{eP1}), it is known that we can discuss equivalently the minimizers of the following constrained
variational minimization problem:
\begin{equation}\label{e1}
 e(a_{1}, a_{2},\beta  ):=\inf_{(u_{1},u_{2})\in \mathcal{M}  } E_{a_{1}, a_{2},\beta}(u_{1},u_{2}),
\end{equation}
where the corresponding energy function $E_{a_{1}, a_{2},\beta}(u_{1},u_{2})$ is given by
\begin{equation}\label{e2}
\begin{split}
E_{a_{1}, a_{2},\beta}(u_{1},u_{2})=&\sum_{i=1}^{2}\int\limits_{\mathbb{R} ^{3} }\left [ u_{i}\sqrt{-\triangle +m^{2} }u_{i}+V_{i}(x)u_{i}^{2} \right ] dx-\sum_{i=1}^{2}\frac{a_{i}}{2}\iint\limits_{\mathbb{R} ^{3}\times \mathbb{R} ^{3}}\frac{u_{i}^{2}(x)u_{i}^{2}(y) }{\left | x-y \right | }dxdy \\
&-\beta \iint\limits_{\mathbb{R} ^{3}\times \mathbb{R} ^{3}}\frac{u_{1}^{2}(x)u_{2}^{2}(y) }{\left | x-y \right | }dxdy, \quad  ({u} _{1},{u} _{2})\in \mathcal{H} _{1}\times \mathcal{H} _{2}.
\end{split}
\end{equation}
Here, the mass constraint $\mathcal{M}$ is defined by
\begin{equation}\label{e3}
\mathcal{M}:= \left \{(u_{1}, u_{2})\in \mathcal{H}_{1}\times \mathcal{H}_{2}:\int\limits_{\mathbb{R} ^{3} }\left(u_{1}  ^{2} + u_{2}  ^{2}\right)dx=1     \right \}
\end{equation}
with
\begin{equation}\label{e4}
\mathcal{H}_{i}= \left \{u\in H^{\frac{1}{2} }(\mathbb{R} ^{3}):\int\limits_{\mathbb{R} ^{3}} V_{i}(x)u^{2}dx< \infty       \right \}, \quad i=1,2,
\end{equation}
and
\begin{equation}\label{e5}
\left \| u \right \|_{\mathcal{H} _{i} }^{2}=\displaystyle\int\limits_{\mathbb{R} ^{3} }\left [ \left | (-\triangle  )^{\frac{1}{4} }u  \right |  ^{2}+V_{i}(x)u^{2}  \right ]dx.
\end{equation}
The classical equation associated to problem (\ref{eP1}) is known as

	\begin{equation} \label{e8}
\sqrt{-\triangle }u(x)+u(x)=\left ( \left | x \right | ^{-1}\ast u^{2}(x)   \right )u(x), \quad  x\in \mathbb{R} ^{3}.
	\end{equation}
Let $Q>0$ be a radially symmetric ground state of (\ref{e8}), it is shown in Lemma 2.2 of \cite{Bona12} that every ground state $Q$ has the decay rate
\begin{equation} \label{e9}
Q(x)=O(\left | x \right | ^{-4} ) \quad {\rm as}\ \left | x \right |  \to \infty.
	\end{equation}
   Denote $a^{\ast }:=\left \| Q \right \|_{2}^{2}$, it is proved in Lemma \ref{le1} in Appendix that all positive ground states of (\ref{e8}) have the same $L^{2}$-norm, that is, $a^{\ast }$ is well-defined. Besides, from (\ref{e11}) in Appendix, $Q$ also satisfies
   \begin{equation} \label{e21}
\int\limits_{\mathbb{R} ^{3} }\left | (-\triangle )^{\frac{1}{4} }Q  \right | ^{2}dx=\int\limits_{\mathbb{R} ^{3} }Q^{2}(x)dx=\frac{1}{2}\iint\limits_{\mathbb{R} ^{3}\times \mathbb{R} ^{3}}\frac{Q^{2}(x)Q^{2}(y) }{\left | x-y \right | }dxdy.
	\end{equation}
   Moreover, it is known as a classical type of Gagliardo-Nirenberg inequality that
   \begin{equation} \label{e17}
\int\limits_{\mathbb{R} ^{3} }\left ( \left | x \right |^{ -1}\ast  u ^{2}    \right )u^{2}dx=\iint\limits_{\mathbb{R} ^{3}\times \mathbb{R} ^{3}}\frac{u^{2}(x)u^{2}(y) }{\left | x-y \right | }dxdy \le \frac{2}{a^{\ast } } \int\limits_{\mathbb{R} ^{3} }\left | (-\triangle )^{\frac{1}{4} }u  \right |^{2}dx\int\limits_{\mathbb{R} ^{3} }u^{2}dx
	\end{equation}
holds for any $u\in H^{\frac{1}{2} }(\mathbb{R} ^{3})$.
Finally, for simplicity, we shall denote the inessential value by $C$, a general positive constant which can change from line to line in the sequel.

  Now, let us turn to introduce the main results. The first result is concerning the existence and non-existence of the minimizers of problem (\ref{e1}).
\begin{theorem} \label{T1}
	Suppose $\left ( \Omega \right )$ holds and denote
\begin{equation} \label{e31}
\beta ^{\ast }:=a ^{\ast }+\sqrt{\left ( a ^{\ast }-a_{1} \right )\left ( a ^{\ast }-a_{2} \right )  },\quad 0< a_{1},a_{2}\le a ^{\ast }.
\end{equation}
Then we have:
\begin{enumerate}[(\romannumeral1)]
\item If $0< a_{1}, a_{2}< a^{\ast } $ and  $0<\beta < \beta ^{\ast } $, then problem (\ref{e1}) has at least one minimizer. Moreover, it holds
$$\lim_{\beta  \to \beta ^{\ast } }  e\left ( a_{1}, a_{2},\beta \right )=e\left ( a_{1}, a_{2},\beta^{\ast } \right )=0, \quad  \ 0< a_{1}, a_{2}< a^{\ast } .$$
\item If  $a_{1}\ge  a^{\ast }$ or $a_{2}\ge  a^{\ast }$ or $\beta \ge  \beta^{\ast }$, then problem (\ref{e1}) has no minimizer.
\end{enumerate}
\end{theorem}
\begin{remark}
	It is worth noticing that the existence result also holds true for the system with repulsive interspecies interaction, i.e., $\beta < 0$, and the system with $\beta = 0$, which can be easily seen in the proof of Theorem \ref{T1} {\rm (i)}.
\end{remark}
Note that, with a mass constraint different from (\ref{e3}) , the authors in \cite{Hkhani15} have investigated problem (\ref{e1}) by constructing an auxiliary functional. Comparing to the pioneer work \cite{Hkhani15}, we address a complete classification involving the three parameters $a_{1}, a_{2}, \beta$ {\re with the mass constraint (\ref{e3}), which leads to the same coefficient $\mu $ of the terms $u_{1}, u_{2}$ in system (\ref{eP1}).} Indeed, the proof of Theorem \ref{T1} is related well to (\ref{e3}) and the following refined Gagliardo-Nirenberg inequality, i.e., for any $({u} _{1},{u} _{2})\in H^{\frac{1}{2} }(\mathbb{R} ^{3})\times H^{\frac{1}{2} }(\mathbb{R} ^{3})$ it holds
\begin{equation} \label{e27}
\begin{split}
&\iint\limits_{\mathbb{R} ^{3}\times \mathbb{R} ^{3}}\frac{ \left ( u_{1}^{2}+u_{2}^{2} \right ) (x)\left ( u_{1}^{2}+u_{2}^{2} \right )(y) }{\left | x-y \right | }dxdy \\
\le &\frac{2}{a^{\ast } } \int\limits_{\mathbb{R} ^{3} }\left(\left | (-\triangle )^{\frac{1}{4} }u_{1}  \right |^{2}+\left | (-\triangle )^{\frac{1}{4} }u_{2}  \right |^{2}\right)dx\int\limits_{\mathbb{R} ^{3} }\left(u_{1}^{2}+u_{2}^{2}\right)dx,
\end{split}
	\end{equation}
which will be described in Lemma \ref{le2} in Appendix. We know that $E_{a_{1}, a_{2},\beta}(u_{1},u_{2})\ge E_{a_{1}, a_{2},\beta}(\left | u_{1} \right | ,\left | u_{2} \right | )$ due to the fact that (\cite{Bona16}, Lemma 4.2)
\begin{equation} \label{e600}
 \int\limits_{\mathbb{R} ^{3} } u  \sqrt{-\triangle +m^{2} }u dx  \ge \int\limits_{\mathbb{R} ^{3} } \left | u \right | \sqrt{-\triangle +m^{2} }\left | u \right |dx, \quad  u\in H^{\frac{1}{2} }(\mathbb{R} ^{3}),
 	\end{equation}
 therefore, without loss of generality, we next restrict the minimizers of problem (\ref{e1}) to nonnegative ones.

  Moreover, let us define a positive constant $\gamma$, i.e.,
  \begin{equation} \label{e48}
\gamma =\frac{\sqrt{a^{\ast }-a_{2}} }{\sqrt{a^{\ast }-a_{1}}+\sqrt{a^{\ast }-a_{2}}  }\in \left ( 0,1 \right ),\ {\rm where}\  0< a_{1}, a_{2}< a^{\ast }.
 \end{equation}
 which plays a fundamental role in the proof of Theorem \ref{T1}. The second main result focuses on the concentration behavior of the nonnegative minimizers for (\ref{e1}) as $\beta \to \beta ^{\ast } $ with a general trapping potential satisfying $\left ( \Omega \right )$.
\begin{theorem} \label{T2}
	Suppose $\left ( \Omega \right )$ holds and let $\left ( u_{1\beta },u_{2\beta }  \right )$ be a nonnegative minimizer of  $e\left ( a_{1}, a_{2},\beta \right )$, where $0< a_{1}, a_{2}< a^{\ast }$ and $0<\beta < \beta ^{\ast } $. Then for any sequence $\left \{ \beta _{k}  \right \} $ with $ \beta _{k} \to \beta ^{\ast }$ as $k\to \infty $, there exists a subsequence of $\left \{ \beta _{k}  \right \} $, still denoted by $\left \{ \beta _{k}  \right \} $, such that for $i=1,2$, each $u_{i\beta _{k}} $ has at least one global maximum point $z_{i\beta _{k}}  $, which satisfies
\begin{equation} \label{e200}
 \lim_{k \to \infty}  z_{1\beta _{k}}= \lim_{k \to \infty}  z_{2\beta _{k}}=x_{0} \quad {\rm for \ some}\ x_{0}\in \mathcal{V},
\end{equation}
where
\begin{equation} \label{e889}
\mathcal{V}:=\left \{ x\in \mathbb{R} ^{3}:V_{1}(x)=V_{2}(x)=0   \right \}.
\end{equation}
 Moreover, we have
\begin{equation} \label{e201}
\begin{cases}
\lim_{k  \to \infty  }\varepsilon _{\beta _{k} }^{\frac{3}{2} }u_{1\beta _{k} }\left ( \varepsilon _{\beta _{k} }x+z_{1\beta _{k} } \right )= \frac{\sqrt{\gamma  }}{\left \| Q \right \| _{2} }Q\left (  x   \right ),\\
\lim_{k  \to \infty  }\varepsilon _{\beta _{k} }^{\frac{3}{2} }u_{2\beta _{k} }\left ( \varepsilon _{\beta _{k} }x+z_{2\beta _{k} } \right )= \frac{\sqrt{1-\gamma  }}{\left \| Q \right \| _{2} }Q\left (  x  \right ),
\end{cases}
\end{equation}
strongly in $H^{\frac{1}{2} }\left ( \mathbb{R} ^{3}  \right ) $, where $\gamma $ is given by (\ref{e48}) and { $Q(x)$ is a positive radially symmetric ground state of (\ref{e8})}, $\varepsilon _{\beta _{k} }>0$ satisfies
\begin{equation} \label{e202}
\varepsilon _{\beta_{k} }:= \left ( \int\limits_{\mathbb{R} ^{3} }\left ( \left | (-\triangle )^{\frac{1}{4} }u_{1\beta_{k} }   \right | ^{2}+\left | (-\triangle )^{\frac{1}{4} }u_{2\beta_{k} }   \right | ^{2} \right ) dx \right )^{-1}\to 0 \quad  {\rm as} \ k\to \infty ,
\end{equation}
\begin{equation} \label{e203}
\lim_{k \to \infty} \frac{\left | z_{1\beta _{k}}-z_{2\beta _{k}}  \right | }{\varepsilon _{\beta _{k}} }=0 ,
\end{equation}
and
\begin{equation} \label{e204}
\lim_{k \to \infty}\frac{\beta ^{\ast }-\beta _{k}  }{\varepsilon _{\beta _{k}} } =0 .
\end{equation}
\end{theorem}
Observe that, there are two non-local terms in our system (pseudo-relativistic operator $\sqrt{-\triangle+m^{2} }$ and Hartree nonlinearities), we can call {\re it} a double non-local problem. It is worth mentioning that the pseudo-relativistic operator $\sqrt{-\triangle +m^{2} }$ is translation invariant but not scaling invariant, which makes our estimates of the energy functional more complicated and challenging. Indeed, denote $w(x )=u(x-x_{0} )$ for any $x_{0} \in \mathbb{R} ^{3}$ and $w,u\in H^{\frac{1}{2} }(\mathbb{R} ^{3}) $, then by Fourier transform, we have
 \begin{equation} \label{e700}
  \begin{split}
  \sqrt{-\triangle +m^{2} }w(x )=&\mathcal{F} ^{-1}\left ( \sqrt{\left |2\pi { s}  \right |  ^{2}+m^{2}}\mathcal{F}w({ s} )   \right )\left ( x  \right )\\
=&\int\limits_{\mathbb{R} ^{3} }e^{i2\pi { s} x } \sqrt{\left | 2\pi { s}  \right | ^{2}+  m^{2}  }  \left ( \int\limits_{\mathbb{R} ^{3} }e^{-i2\pi { s} y}w\left (y \right )dy\right ) d{ s} \\
=& \displaystyle \int\limits_{\mathbb{R} ^{3} }e^{i2\pi { s} x } \sqrt{\left | 2\pi { s}  \right | ^{2}+  m^{2}  }  \left (\int\limits_{\mathbb{R} ^{3} }e^{-i2\pi { s} y}u\left (y-x_{0}  \right )dy \right ) d{ s} \\
=& \displaystyle \int\limits_{\mathbb{R} ^{3} }e^{i2\pi { s} (x-x_{0}) } \sqrt{\left | 2\pi { s}  \right | ^{2}+  m^{2}  } \left ( \int\limits_{\mathbb{R} ^{3} }e^{-i2\pi { s} \left ( y-x_{0} \right ) }u\left (y-x_{0}  \right )dy \right ) d{ s} \\
=&\mathcal{F} ^{-1}\left ( \sqrt{\left |2\pi { s}  \right |  ^{2}+m^{2}}\mathcal{F}u({ s} )   \right )\left ( x -x_{0}  \right )\\
=& \sqrt{-\triangle +m^{2} }u(x-x_{0}),
\end{split}
   \end{equation}
which implies the translation invariance holds. On the other hand, let $w(x)=u(\lambda  x)$  for any $ \lambda \in \mathbb{R}$ and $w,u\in H^{\frac{1}{2} }(\mathbb{R} ^{3}) $, then we can obtain by taking $\xi =\lambda y$ and $\delta =\frac{{ s}}{\lambda }$ that
 \begin{equation} \label{e701}
  \begin{split}
  \sqrt{-\triangle +m^{2} }w(x )=&\mathcal{F} ^{-1}\left ( \sqrt{\left |2\pi { s}  \right |  ^{2}+m^{2}}\mathcal{F}w({ s} )   \right )\left ( x  \right )\\
=&\int\limits_{\mathbb{R} ^{3} }e^{i2\pi { s} x } \sqrt{\left | 2\pi { s}  \right | ^{2}+  m^{2}  }  \left ( \int\limits_{\mathbb{R} ^{3} }e^{-i2\pi { s} y}w\left (y \right )dy \right ) d{ s} \\
=&\int\limits_{\mathbb{R} ^{3} }e^{i2\pi { s} x } \sqrt{\left | 2\pi { s}  \right | ^{2}+  m^{2}  }  \left ( \int\limits_{\mathbb{R} ^{3} }e^{-i2\pi { s} y}u\left (\lambda y \right )dy \right ) d{ s} \\
=&\lambda \int\limits_{\mathbb{R} ^{3} }e^{i2\pi \delta  \lambda  x } \sqrt{\left | 2\pi \delta    \right | ^{2}+  \lambda ^{-2} m^{2}  }  \left ( \int\limits_{\mathbb{R} ^{3} }e^{-i2\pi \delta  \xi }u\left (\xi  \right )d\xi  \right ) d\delta\\
=&\lambda \mathcal{F} ^{-1}\left ( \sqrt{\left |2\pi \delta   \right |  ^{2}+\lambda ^{-2}m^{2}}\mathcal{F}u(\delta )   \right )\left ( \lambda x   \right )  \\
=&\lambda \sqrt{-\triangle +\lambda ^{-2}m^{2} }u(\lambda x),
  \end{split}
   \end{equation}
which indicates that the scaling invariance of $\sqrt{-\triangle +m^{2} }$ fails with multiplying an extra scaling coefficient $\lambda$. Also, it is easy to verify that the Hartree nonlinear terms have similar translation and scaling properties. Therefore, our aim is to develop some new techniques to overcome these difficulties arising from the double non-local terms according to the Fourier transform developed in \cite{Bona16} and the classical Hardy-Littlewood-Sobolev inequality.

Through the results of reference  \cite{Du22} and its proof we found that the existence of the global maximum point of each nonnegative minimizer is based heavily on the tool of non-local De Giorgi-Nash-Moser theory introduced in \cite{Imbesi} and \cite{Kim}. But unfortunately we can not obtain the uniqueness of the global maximum point for our problem (\ref{e1}) since the deficiency of classical elliptic regularity theory in fractional Sobolev space, which plays a crucial role in proving the smoothness of the minimizer sequences. Therefore, this can be an open question for future research.

In order to give a further analysis on the concentration behavior of nonnegative minimizers, in the following, we basically assume the trapping potential $V_{i}(x)$ of the form
\begin{equation} \label{e212}
V_{i}(x)=h_{i}(x)\prod_{j=1}^{m_{i} }\left | x-x_{ij} \right | ^{q_{ij} },\  q_{ij}>0 \ {\rm and}\ m_{i}\in \mathbb{N} ^{+},\ i=1,2,
\end{equation}
where $h_{i}(x)\in C_{loc}^{\alpha }(\mathbb{R} ^{3} )$ satisfies $\frac{1}{C} <h_{i}(x)< C$ for some positive constant $C$ and $x_{ij}\ne x_{ik}$ for $j\ne k$. Also, we suppose,  without losing generality, that there is a positive integer $b$ satisfying $1\le b\le \min \left \{m_{ 1},m_{ 2}  \right \}  $ such that
\begin{equation} \label{e213}
\begin{cases}
x_{1j}=x_{2j}:=x_{j},  \ \quad {\rm where}\ j=1,\cdots ,b; \\
x_{1j}\ne x_{2k}, \quad \quad \ \ \ \ \ \ {\rm where}\ j\in \left \{ b+1,\cdots ,m_{1}  \right \} \ {\rm and}\ k\in \left \{b+1,\cdots,m_{2}   \right \}.
\end{cases}
\end{equation}
Notice that, (\ref{e213}) indicates
\begin{equation} \label{e214}
\mathcal{V}=\left \{ x\in \mathbb{R} ^{3}:V_{1}(x)=V_{2}(x)=0   \right \}=\left \{ x_{1},\cdots , x_{b}   \right \}.
\end{equation}
Additionally, we define $\lambda _{j} $ for $j=1,\cdots ,b  $ by
\begin{equation} \label{e215}
\lambda _{j}:=
\begin{cases}
\lim_{x \to x_{j} }\frac{\gamma V_{1}(x)}{\left | x-x_{j} \right |^{q_{1j}} }\displaystyle \int\limits_{\mathbb{R}^{3} }\left | { y} \right |^{q_{1j}}Q^{2}({ y} )d{ y}   \ \  \quad \quad \quad \quad \quad \quad \  {\rm if }\quad  q_{1j}<q_{2j},  \\
\lim_{x \to x_{j} }\frac{\gamma V_{1}(x)+(1-\gamma )V_{2}(x)  }{\left | x-x_{j} \right |^{q_{1j}} }\displaystyle \int\limits_{\mathbb{R}^{3} }\left | { y}  \right |^{q_{1j}}Q^{2}({ y} )d{ y}   \quad \quad \quad \ \  {\rm if } \quad q_{1j}=q_{2j}, \\
\lim_{x \to x_{j} }\frac{(1-\gamma ) V_{2}(x)}{\left | x-x_{j} \right |^{q_{2j}} }\displaystyle \int\limits_{\mathbb{R}^{3} }\left | { y}  \right |^{q_{2j}}Q^{2}({ y} )d{ y}  \ \ \quad \quad \quad \quad \quad \ \   {\rm if }\quad  q_{1j}>q_{2j},
\end{cases}
\end{equation}
where $\gamma $ is given by (\ref{e48}). Define
\begin{equation} \label{e216}
q_{j}:=\min \left \{ q_{1j},q_{2j} \right \},\ j=1,\cdots , b;\quad q_{0}:=\max_{1\le j\le b}\  q_{j},
\end{equation}
so that
\begin{equation} \label{e217}
Z:=\left \{ x_{j}:q_{j}=q_{0},j=1,\cdots ,b  \right \} \subset \mathcal{V} .
\end{equation}
Set
\begin{equation} \label{e218}
\lambda _{0}:=\min_{j\in \Gamma }\lambda _{j},\quad {\rm where}\ \Gamma :=\left \{ j:x_{j}\in Z  \right \} ,
\end{equation}
and let
\begin{equation} \label{e219}
Z_{0}:=\left \{ x_{j}\in Z:\lambda _{j}=\lambda _{0}    \right \}
\end{equation}
denote the set of the flattest common minimum points of $V_{i}(x)   $ for $i=1,2$. Our third significant result is about the precise concentration behavior of nonnegative minimizers for (\ref{e1}) with a polynomial trapping potential defined above.
\begin{theorem} \label{T3}
	Suppose $ V_{i}(x) \left (i=1,2\right )$ satisfy (\ref{e212}) and (\ref{e213}), let $\left ( u_{1\beta },u_{2\beta }  \right )$ be a nonnegative minimizer of $e\left ( a_{1}, a_{2},\beta \right )$, where $0< a_{1}, a_{2}< a^{\ast }$ and $0<\beta < \beta ^{\ast } $. Moreover, if it holds $0<q_{0}<1$ when $m\ne 0$ and $0<q_{0}<\frac{5}{2} $ when $m=0$,
then for any sequence $\left \{ \beta _{k}  \right \} $ with $ \beta _{k} \to \beta ^{\ast }$ as $k\to \infty $, there exists a subsequence of $\left \{ \beta _{k}  \right \} $, still denoted by $\left \{ \beta _{k}  \right \} $, such that for $i=1,2$, each $u_{i\beta _{k}} $ has at least one global maximum point $z_{i\beta _{k}}  $, which satisfies
\begin{equation} \label{e220}
 \lim_{k \to \infty}  \frac{z_{1\beta _{k}}-x_{j_{0}}}{\epsilon _{\beta _{k}}} = \lim_{k \to \infty} \frac{z_{2\beta _{k}}-x_{j_{0}}}{\epsilon _{\beta _{k}}}=0 \quad {\rm for \ some} \ x_{j_{0}}\in Z_{0},
\end{equation}
where $\epsilon _{\beta _{k}}>0$ is given by
\begin{equation} \label{e221}
\epsilon _{\beta _{k}}:=\left [ \frac{2\gamma (1-\gamma )}{q_{0}\lambda_{0} } \left ( \beta ^{\ast }-\beta _{k}   \right )  \right ]^{\frac{1}{q_{0}+1} } .
\end{equation}
 Moreover, we have
\begin{equation} \label{e722}
\begin{cases}
\lim_{k  \to \infty  }\epsilon  _{\beta _{k} }^{\frac{3}{2} }u_{1\beta _{k} }\left ( \epsilon  _{\beta _{k} }x+z_{1\beta _{k} } \right )= \frac{\sqrt{\gamma  }}{\left \| Q \right \| _{2} }Q\left (  x   \right ),\\
\lim_{k  \to \infty  }\epsilon  _{\beta _{k} }^{\frac{3}{2} }u_{2\beta _{k} }\left ( \epsilon  _{\beta _{k} }x+z_{2\beta _{k} } \right )= \frac{\sqrt{1-\gamma  }}{\left \| Q \right \| _{2} }Q\left (  x  \right ),
\end{cases}
\end{equation}
strongly in $H^{\frac{1}{2} }\left ( \mathbb{R} ^{3}  \right ) $ and
\begin{equation} \label{e223}
\lim_{k \to \infty} \epsilon  _{\beta_{k} }  \int\limits_{\mathbb{R} ^{3} }\left ( \left | (-\triangle )^{\frac{1}{4} }u_{1\beta_{k} }   \right | ^{2}+\left | (-\triangle )^{\frac{1}{4} }u_{2\beta_{k} }   \right | ^{2} \right ) dx =1.
\end{equation}
\end{theorem}
It is interesting to note from the proof of Theorem \ref{T2} that the behavior of minimizers {\re does} depend on the local character of $V(x)$ at the minima. By taking the trapping potential to be polynomial form, a detailed description of the concentration behavior of the nonnegative minimizers for (\ref{e1}) is presented in Theorem \ref{T3}, in particular, an explicit blow up rate is given in (\ref{e223}). Moreover, (\ref{e220}) implies that the nonnegative minimizers of (\ref{e1}) will actually concentrate at a flattest common minima of $V_{i}(x) (i=1,2)$ as $\beta _{k} \to \beta ^{\ast } $. For the proof of Theorem \ref{T3}, some methods in reference \cite{Du22} are invalid when estimating the energy functional, due to the fact that the ground state concerning fractional Laplace operator has a polynomial decay, which is quite different from the exponential one in the classical Laplacian system.

 The paper is organized as follows. In section 2, we present a complete classifying on existence and non-existence of the minimizers for (\ref{e1}). Section 3 is focused on the analysis of the concentration behavior of the minimizers for (\ref{e1}) with general trapping potentials as $\beta \to \beta ^{\ast }$. In Section 4, we are devoted to estimating the associated functional precisely under polynomial trapping potentials and complete the proof of Theorem \ref{T3}. In Appendix, we provide the detailed proofs of two results used in our study as a complement.

\section{Existence and non-existence results}
In this section, we investigate the existence and non-existence of the minimizers for (\ref{e1}). To this end, we first recall the basic tools that will be used later.

\begin{lemma}\label{b1}( \cite{Hkhani15}, Lemma 2.1) Suppose $V_{i}\in L_{loc }^{\infty}\left ( \mathbb{R} ^{3}  \right )$ and $\lim_{\left | x \right |  \to \infty } V_{i}(x)= \infty , i=1,2$. Then the embedding $\mathcal{H} _{1} \times \mathcal{H} _{2}\hookrightarrow L^{q}\left ( \mathbb{R} ^{3}  \right )\times L^{q}\left ( \mathbb{R} ^{3}  \right )$ is compact for $2\le q< 3$.
\end{lemma}
The following two inequalities play the key roles in our proof.
\begin{lemma}\label{b2}(Hardy-Littlewood-Sobolev inequality, \cite{Galewski}, Theorem 4.3) Let $p,r>1$ and $0<t<3$ with $\frac{1}{p}+\frac{1}{r}+ \frac{t}{3}=2$. Let $f\in L^{p}\left ( \mathbb{R} ^{3}  \right )$ and  $h\in L^{r}\left ( \mathbb{R} ^{3}  \right )$. Then there exists a sharp constant $C(t,p)$, such that
 $$\left | \ \  \iint\limits_{\mathbb{R} ^{3}\times \mathbb{R} ^{3}}\frac{f(x)h(y)}{\left | x-y \right | ^{t} }dxdy \right | \le C(t,p)\left \| f \right \| _{p}\left \| h \right \| _{r}  ,$$
\end{lemma}

\begin{lemma}\label{b3}( \cite{Hkhani15}, Lemma 2.3) Let $u_{1},u_{2}\in L^{\frac{12}{5} }\left ( \mathbb{R} ^{3}  \right )$. Then,
$$\iint\limits_{\mathbb{R} ^{3}\times \mathbb{R} ^{3} }\frac{u_{1}^{ 2}(x)u_{2}^{ 2}(y) }{\left | x-y \right | }dxdy\le\left (\ \  \iint\limits_{\mathbb{R} ^{3}\times \mathbb{R} ^{3} }\frac{u_{1}^{ 2}(x)u_{1}^{ 2}(y) }{\left | x-y \right | }dxdy\iint\limits_{\mathbb{R} ^{3}\times \mathbb{R} ^{3} }\frac{u_{2}^{ 2}(x)u_{2}^{ 2}(y) }{\left | x-y \right | }dxdy \right ) ^{\frac{1}{2} },$$
and the equality holds if and only if $u_{1} ^{2}(x)=\kappa u_{2} ^{2}(x)$ for some $\kappa > 0$.
\end{lemma}
For simplicity, we define a symmetric bilinear form
$$D\left ( u_{1},u_{2}  \right ) :=\iint\limits_{\mathbb{R} ^{3}\times \mathbb{R} ^{3} }\frac{u_{1}(x)u_{2}(y) }{\left | x-y \right | }dxdy \quad {\rm for \ any} \ u_{1},u_{2}\in L^{\frac{6}{5} }\left ( \mathbb{R} ^{3}  \right ).$$
Taking $p=r=\frac{6}{5}$ in Lemma \ref{b2}, we thus obtain
\begin{equation} \label{e222}
D\left (u_{1}^{ 2} ,u_{2}^{ 2}  \right ) =\iint\limits_{\mathbb{R} ^{3}\times \mathbb{R} ^{3} }\frac{u_{1}^{ 2}(x)u_{2}^{ 2}(y) }{\left | x-y \right | }dxdy\le C_{0}\left \|  u_{1}\right \| _{L^{\frac{12}{5} }}^{2} \left \|  u_{2}\right \| _{L^{\frac{12}{5} }}^{2}  \quad {\rm for \ any} \  u_{1},u_{2}\in L^{\frac{12}{5} }\left ( \mathbb{R} ^{3}  \right ).
\end{equation}
Based on the above notations, we rewrite the associated energy functional $E_{a_{1}, a_{2},\beta}(u_{1},u_{2})$ defined in (\ref{e2}) as
\begin{equation}\label{e29}
\begin{split}
E_{a_{1}, a_{2},\beta}(u_{1},u_{2})=&\sum_{i=1}^{2}\int\limits_{\mathbb{R} ^{3} }\left [ u_{i}\sqrt{-\triangle +m^{2} }u_{i}+V_{i}(x)u_{i}^{2} \right ]dx\\
&-\frac{1}{2}\left [ a_{1}D\left ( u_{1} ^{2},u_{1} ^{2}  \right )+ a_{2}D\left ( u_{2} ^{2},u_{2} ^{2}  \right )+2\beta D\left ( u_{1} ^{2},u_{2} ^{2}  \right )  \right ]
\end{split}
\end{equation}
for any $(u_{1}, u_{2})\in \mathcal{H}_{1}\times \mathcal{H}_{2}$.
We next introduce the convergence of the convolution terms and give the proof.
\begin{lemma}\label{b4} If $\left ( u_{1n},u_{2n} \right )\to \left ( u_{10},u_{20} \right )$  strongly in $L^{\frac{12}{5} }\left ( \mathbb{R} ^{3}  \right ) \times L^{\frac{12}{5} }\left ( \mathbb{R} ^{3}  \right )$, then $$\lim_{n \to \infty }D\left ( u_{1n}^{2},u_{2n}^{2}  \right )=D\left ( u_{10}^{2},u_{20}^{2}  \right ),$$
$$\lim_{n \to \infty }D\left ( u_{in}^{2},u_{in}^{2}  \right )=D\left ( u_{i0}^{2},u_{i0}^{2}  \right ), \quad i=1,2.$$
\end{lemma}
\begin{proof}We only prove the convergence for the mixed convolution term, the proofs for the other two are similar and we ignore them here.
Based on the convergence of $\left ( u_{1n},u_{2n} \right )$, it follows from the Hardy-Littlewood-Sobolev inequality and the H\"{o}lder inequality that we have
\begin{equation} \label{e30}
\begin{split}
&\left | D\left ( u_{1n}^{2},u_{2n}^{2}  \right )-D\left ( u_{10}^{2},u_{20}^{2}  \right ) \right | \\
=&\left | \int\limits_{\mathbb{R} ^{3} }\left ( \left | x \right |^{ -1}\ast  u_{1n}   ^{2}    \right )u_{2n}^{2}dx-\int\limits_{\mathbb{R} ^{3} }\left ( \left | x \right |^{ -1}\ast  u_{10}    ^{2}    \right )u_{20}^{2}dx \right |\\
\le& \left | \int\limits_{\mathbb{R} ^{3} }\left ( \left | x \right |^{ -1}\ast  u_{1n}   ^{2}    \right )\left ( u_{2n}^{2}-u_{20}^{2} \right ) dx \right |+\left | \int\limits_{\mathbb{R} ^{3} }\left ( \left | x \right |^{ -1}\ast \left ( u_{1n}^{2} -u_{10}^{2} \right )         \right )u_{20}^{2}dx \right |\\
\le& C\left ( \left \| u_{1n}^{2} \right \| _{\frac{6}{5} }\left \| u_{2n}^{2}-u_{20}^{2} \right \| _{\frac{6}{5} }+\left \| u_{1n}^{2}-u_{10}^{2} \right \| _{\frac{6}{5} } \left \| u_{20}^{2} \right \| _{\frac{6}{5} }  \right )\\
\le& C\left ( \left \| u_{1n} \right \| _{\frac{12}{5} }^{2}\left \| u_{2n}-u_{20} \right \| _{\frac{12}{5} }\left \| u_{2n}+u_{20} \right \| _{\frac{12}{5} }+\left \| u_{1n}-u_{10} \right \| _{\frac{12}{5} }\left \| u_{1n}+u_{10} \right \| _{\frac{12}{5} } \left \| u_{20} \right \| _{\frac{12}{5} }^{2}  \right )\\
\le& C\left ( \left \| u_{2n}-u_{20} \right \| _{\frac{12}{5} }+\left \| u_{1n}-u_{10} \right \| _{\frac{12}{5} }  \right ).
\end{split}
\end{equation}
The proof is finished by letting $n\to \infty $ in (\ref{e30}).
\end{proof}
\vspace{0.5cm}
\noindent{\bf Proof of Theorem \ref{T1}:} Our proof is divided into three parts.
\begin{enumerate}[(\romannumeral1)]
\item In the case $0< a_{1}, a_{2}< a^{\ast } $ and $0< \beta < \beta ^{\ast } $, we claim that (\ref{e1}) has at least one minimizer.
 \end{enumerate}
First, we deduce that $E_{a_{1}, a_{2},\beta}(u_{1},u_{2})$ is bounded from below in $\mathcal{M}$. Indeed, since $ \sqrt{-\triangle +m^{2} }\ge \sqrt{-\triangle }$, i.e.,
 \begin{equation} \label{e33}
\left (  \sqrt{-\triangle +m^{2} }u,u \right ) \ge \left ( \sqrt{-\triangle }u,u \right ), \quad  u\in H^{\frac{1}{2} }(\mathbb{R} ^{3} ),
\end{equation}
by rewriting the energy functional $E_{a_{1}, a_{2},\beta}(u_{1},u_{2})$ { in (\ref{e29})} and applying the refined Gagliardo-Nirenberg inequality (\ref{e27}) and Lemma \ref{b3}, we therefore obtain
  \begin{equation} \label{e51}
  \begin{split}
E_{a_{1}, a_{2},\beta}(u_{1},u_{2}) \ge& \int\limits_{\mathbb{R} ^{3} }\left(\left | (-\triangle )^{\frac{1}{4} }u_{1}  \right |^{2}+\left | (-\triangle )^{\frac{1}{4} }u_{2}  \right |^{2}\right )dx-\frac{a^{\ast }}{2}D\left ( u_{1}^{2}+u_{2}^{2},u_{1}^{2}+u_{2}^{2} \right )\\
&+\frac{1}{2}\left [ \sqrt{\left ( a^{\ast }-a_{1} \right )D\left ( u_{1}^2,u_{1}^2 \right )    }-\sqrt{\left ( a^{\ast }-a_{2} \right )D\left ( u_{2}^2,u_{2}^2 \right )   }    \right ]^2 \\
&+\sqrt{\left ( a^{\ast }-a_{1} \right )\left ( a^{\ast }-a_{2} \right )   } \left [  \sqrt{D\left ( u_{1}^2,u_{1}^2 \right )D\left ( u_{2}^2,u_{2}^2 \right )}-D\left ( u_{1}^2,u_{2}^2 \right )  \right ] \\
&+\left ( \beta ^{\ast }  -\beta \right )  D\left ( u_{1}^2,u_{2}^2 \right )+\int\limits_{\mathbb{R} ^{3} }\left ( V_{1}(x)u_{1}^{2}+V_{2}(x)u_{2}^{2} \right )dx\\
\ge& 0
    \end{split}
 \end{equation}
 holds for any $\left ( u_{1} ,u_{2}  \right )\in \mathcal{M}$. It then follows that $E_{a_{1}, a_{2},\beta}(u_{1},u_{2})$ is bounded from below in $\mathcal{M}$.
 For any given $0< a_{1}, a_{2}< a^{\ast } $ and $0< \beta < \beta ^{\ast } $, it is easily known that there exists an $a\in \left ( 0,\infty  \right )$ such that $$max\left \{ a_{1}, a_{2}  \right \}\le a< a^{\ast },\quad \beta < a+  \sqrt{\left ( a-a_{1} \right )\left ( a-a_{2} \right )  }.$$
 Let $\left \{ \left ( u_{1n},u_{2n}  \right )  \right \}\subset \mathcal{M}$ be a minimizing sequence of problem (\ref{e1}), then
\begin{equation} \label{e32}
\int\limits_{\mathbb{R} ^{3} }\left ( u_{1n}  ^{2}+\ u_{2n}  ^{2}\right )dx=1,\quad \lim_{n \to \infty } E_{a_{1}, a_{2},\beta}(u_{1n},u_{2n})=e\left ( a_{1}, a_{2},\beta \right ).
\end{equation}
Similarly as (\ref{e51}), one can deduce that
\begin{equation} \label{e34}
\begin{split}
E_{a_{1}, a_{2},\beta}(u_{1n},u_{2n})\ge& \int\limits_{\mathbb{R} ^{3} }\left(\left | (-\triangle )^{\frac{1}{4} }u_{1n}  \right |^{2}+\left | (-\triangle )^{\frac{1}{4} }u_{2n}  \right |^{2}\right)dx-\frac{a}{2}D\left ( u_{1n}^{2}+u_{2n}^{2},u_{1n}^{2}+u_{2n}^{2} \right )\\
&+\frac{1}{2}\left [ \sqrt{\left ( a-a_{1} \right )D\left ( u_{1n}^2,u_{1n}^2 \right )    }-\sqrt{\left ( a-a_{2} \right )D\left ( u_{2n}^2,u_{2n}^2 \right )   }    \right ]^2 \\
&+\sqrt{\left ( a-a_{1} \right )\left ( a-a_{2} \right )   } \left [  \sqrt{D\left ( u_{1n}^2,u_{1n}^2 \right )D\left ( u_{2n}^2,u_{2n}^2 \right )}-D\left ( u_{1n}^2,u_{2n}^2 \right )  \right ] \\
&+\left ( a+\sqrt{\left ( a-a_{1} \right )\left ( a-a_{2} \right )   }  -\beta \right )  D\left ( u_{1n}^2,u_{2n}^2 \right )\\
&+\int\limits_{\mathbb{R} ^{3} }\left ( V_{1}(x)u_{1n}^{2}+V_{2}(x)u_{2n}^{2} \right )dx \\
\ge& \left ( 1-\frac{a}{a^{\ast } } \right )  \displaystyle\int\limits_{\mathbb{R} ^{3} }\left ( \left | (-\triangle )^{\frac{1}{4} }u_{1n}  \right |^{2}+\left | (-\triangle )^{\frac{1}{4} }u_{2n}  \right |^{2} \right ) dx\\
&+\int\limits_{\mathbb{R} ^{3} }\left ( V_{1}(x)u_{1n}^{2}+V_{2}(x)u_{2n}^{2} \right )dx.
\end{split}
\end{equation}
Let $n\to \infty $ in the both side of the above inequality, we then conclude from (\ref{e32}) that the minimizing sequence $\left \{ \left ( u_{1n},u_{2n}  \right )  \right \}$ is bounded in $\mathcal{H} _{1}\times \mathcal{H} _{2}$. From Lemma \ref{b1}, passing to a subsequence if necessary, we infer there exists $\left ( u_{10},u_{20}  \right )\in \mathcal{H} _{1}\times \mathcal{H} _{2}$ such that
\begin{equation} \label{e35}
\begin{cases}
\left ( u_{1n},u_{2n}  \right ) \rightharpoonup  \left ( u_{10},u_{20}  \right ) \quad {\rm weakly \ in}\  \mathcal{H} _{1}\times \mathcal{H} _{2},\\
\left ( u_{1n},u_{2n}  \right ) \to \left ( u_{10},u_{20}  \right ) \quad {\rm strongly \ in}\  L^{q}\left ( \mathbb{R} ^{3}  \right )\times L^{q}\left ( \mathbb{R} ^{3}  \right ), \  2\le q< 3.
\end{cases}
\end{equation}
Therefore, we conclude that $$\lim_{n \to \infty } \int\limits_{\mathbb{R} ^{3} }\left ( u_{1n}   ^{2}+ u_{2n}   ^{2}\right )dx=1=\int\limits_{\mathbb{R} ^{3} }\left ( u_{10}  ^{2}+ u_{20}   ^{2}\right )dx, \ {\rm i.e.,}\ \left ( u_{10},u_{20} \right ) \in \mathcal{M}.$$
Furthermore, it follows from the weak lower semi-continuity of $\left (  \sqrt{-\triangle +m^{2} }u,u \right )$ and Lemma \ref{b4} that
$$ e\left ( a_{1}, a_{2}, \beta   \right )\le E_{a_{1}, a_{2}, \beta}\left ( u_{10}, u_{20} \right )\le \lim_{n \to \infty } E_{a_{1}, a_{2}, \beta}\left ( u_{1n}, u_{2n} \right )=e\left ( a_{1}, a_{2}, \beta   \right ),$$
which implies that $\left ( u_{10}, u_{20} \right )$ is a minimizer of $e\left ( a_{1}, a_{2}, \beta   \right )$.

\begin{enumerate}[(\romannumeral2)]
\item In the case  $a_{1}> a^{\ast }$ or $a_{2}> a^{\ast }$ or $\beta> \beta ^{\ast }$, we infer that (\ref{e1}) has no minimizers .
 \end{enumerate}
 Let $\varphi \in C_{0}^{\infty }\left ( \mathbb{R} ^{3}  \right )$ be a cut-off function satisfying
 $$\varphi \left ( x \right )\equiv 1\ \ {\rm for}\ \left | x \right |\le 1, \ \varphi \left ( x \right )\equiv0 \ \ {\rm for}\  \left | x \right |\ge 2,\ 0\le \varphi \left ( x \right )\le 1 \ {\rm for}\ \forall x\in \mathbb{R} ^{3}.$$
 Besides, for any $\sigma > 0$, set
 \begin{equation} \label{e36}
\begin{cases}
\phi _{1\sigma }\left ( x \right ):=\frac{\sqrt{\theta }A_{\sigma }  }{\left \| Q \right \| _{2} }\sigma ^{\frac{3}{2} }\varphi \left ( x-x_{0}  \right )Q\left ( \sigma \left ( x-x_{0} \right )  \right ),\\
\phi _{2\sigma }\left ( x \right ):=\frac{\sqrt{1-\theta }A_{\sigma }  }{\left \| Q \right \| _{2} }\sigma ^{\frac{3}{2} }\varphi \left ( x-x_{0}  \right )Q\left ( \sigma \left ( x-x_{0} \right )  \right ),
\end{cases}
\end{equation}
where $A_{\sigma }>0$ is chosen such that $\left ( \phi _{1\sigma },\phi _{2\sigma } \right )\in \mathcal{M} $, $\theta \in \left [ 0,1 \right ]$ and $x_{0}\in \mathbb{R} ^{3}$ will be determined latter. Since $\left ( \phi _{1\sigma },\phi _{2\sigma } \right )\in \mathcal{M} $, we have $\displaystyle \int\limits_{\mathbb{R} ^{3} }\left (\left | \phi _{1\sigma } \right | ^{2}+\left | \phi _{2\sigma } \right | ^{2}\right )dx=1 $, that is
\begin{equation} \label{e37}
1=\int\limits_{\mathbb{R} ^{3} }\frac{A_{\sigma }^2  }{\left \| Q \right \| _{2}^2 }\sigma ^{3 }\varphi^2 \left ( x-x_{0}  \right )Q^2\left ( \sigma \left ( x-x_{0} \right )  \right )dx=\int\limits_{\mathbb{R} ^{3} }\frac{A_{\sigma }^2  }{\left \| Q \right \| _{2}^2 }\varphi^2 \left ( \sigma ^{-1} x  \right )Q^2\left ( x  \right )dx.
\end{equation}
From (\ref{e9}) and (\ref{e37}), one can deduce
\begin{equation} \label{e38}
\begin{split}
\left | 1-A_{\sigma }^{2}  \right |=&\left | \int\limits_{\mathbb{R} ^{3} }\frac{A_{\sigma }^2  }{\left \| Q \right \| _{2}^2 }Q^{2}\left ( x \right )\left ( \varphi ^2\left ( \sigma ^{-1} x \right )-1  \right )dx     \right |\\
\le& \int\limits_{\mathbb{R} ^{3}\setminus B_{\sigma }  }\left | \frac{A_{\sigma }^2  }{\left \| Q \right \| _{2}^2 }Q^{2}\left ( x \right )\left ( \varphi ^2\left ( \sigma ^{-1} x \right )-1  \right ) \right | dx\\
\le& C\int\limits_{\mathbb{R} ^{3}\setminus B_{\sigma }  }\left | x\right |^{-8} dx\le C\sigma^{-5},
\end{split}
\end{equation}
as $\sigma$ large enough, where $B_{\sigma }:=\left \{ x\in \mathbb{R} ^{3}\mid \left | x \right |< \sigma    \right \} $.
Hence we obtain
\begin{equation} \label{e39}
1-O\left ( \sigma ^{-5}  \right ) \le A_{\sigma }^2 \le 1+ O\left ( \sigma ^{-5}  \right )\quad {\rm as}\ \sigma \to \infty.
\end{equation}
Now, let us turn to estimate the non-local terms in $E_{a _{1}, a_{2},\beta}(\phi _{1\sigma },\phi _{2\sigma })$, i.e., the terms with pseudo-relativistic operator $\sqrt{-\triangle+m^{2} }$ and the terms with Hartree nonlinearities. Denote $\varphi _{\sigma }= \varphi (\sigma ^{-1}x )$ and apply (\ref{e700}) and (\ref{e701}) to the test function $\phi _{1\sigma }$ with scaling coefficient $\sigma$,  we thus obtain
 \begin{equation} \label{e702}
  \begin{split}
  \int\limits_{\mathbb{R} ^{3} } \phi _{1\sigma }\sqrt{-\triangle +m^{2} }\phi _{1\sigma }dx=\frac{\theta A_{\sigma }^{2}}{\left \| Q \right \| _{2}^2}\sigma \int\limits_{\mathbb{R} ^{3} }\varphi _{\sigma} Q(x)\sqrt{-\triangle +\sigma  ^{-2}m^{2} }\varphi _{\sigma}Q(x)dx.
    \end{split}
   \end{equation}
Based on the operator inequality $\sqrt{-\triangle +\sigma  ^{-2}m^{2} }\le \sqrt{-\triangle  }+\frac{1}{2}\sigma ^{-2}m^{2}\left ( -\triangle \right ) ^{-\frac{1}{2} } $ and notice the estimates (see \cite{Bona16})
 \begin{equation} \label{e703}
 \int\limits_{\mathbb{R} ^{3} }\varphi _{\sigma}Q(x)\sqrt{-\triangle }\varphi _{\sigma}Q(x)dx\le\int\limits_{\mathbb{R} ^{3} }Q(x)\sqrt{-\triangle }Q(x)dx+ C\sigma ^{-\frac{7}{2} },
  \end{equation}
  and
  \begin{equation} \label{e704}
 \int\limits_{\mathbb{R} ^{3} }\varphi _{\sigma}Q(x)\left ( -\triangle \right )^{-\frac{1}{2} }  \varphi _{\sigma}Q(x)dx\le\int\limits_{\mathbb{R} ^{3} }Q(x)\left ( -\triangle \right )^{-\frac{1}{2} }Q(x)dx,
  \end{equation}
one can infer from (\ref{e21}), (\ref{e39})-(\ref{e704}) that
 \begin{equation} \label{e40}
 \begin{split}
 \int\limits_{\mathbb{R} ^{3} } \phi _{1\sigma }\sqrt{-\triangle +m^{2} }\phi _{1\sigma }dx \le& \frac{\theta \sigma }{\left \| Q \right \| _{2} ^{2} } \int\limits_{\mathbb{R} ^{3} }Q\sqrt{-\triangle }Qdx+O\left ( \sigma ^{-\frac{5}{2} }  \right )+  m^{2}O\left ( \sigma ^{-1}  \right )\\
 =&\theta \sigma + O\left ( \sigma ^{-\frac{5}{2} }  \right )+  m^{2}O\left ( \sigma ^{-1}  \right ).
 \end{split}
 \end{equation}
Apply the same steps to the test function $\phi  _{2\sigma }$, one can also obtain
 \begin{equation} \label{e41}
  \begin{split}
 \int\limits_{\mathbb{R} ^{3} } \phi _{2\sigma }\sqrt{-\triangle +m^{2} }\phi _{2\sigma }dx \le& \frac{\left ( 1-\theta  \right ) \sigma }{\left \| Q \right \| _{2} ^{2} } \int\limits_{\mathbb{R} ^{3} }Q\sqrt{-\triangle }Qdx+O\left ( \sigma ^{-\frac{5}{2} }  \right )+  m^{2}O\left ( \sigma ^{-1}  \right )\\
 =&\left ( 1-\theta \right )  \sigma + O\left ( \sigma ^{-\frac{5}{2} }  \right )+  m^{2}O\left ( \sigma ^{-1}  \right ).
  \end{split}
 \end{equation}
 For Hartree nonlinear terms, we first claim
 \begin{equation} \label{e708}
 \frac{ 2\theta ^{2}\sigma} { a^{\ast}}-O\left ( \sigma ^{-4}  \right ) \le D\left ( \phi _{1\sigma } ^2,\phi _{1\sigma }^2 \right )\le \frac{ 2\theta ^{2}\sigma} { a^{\ast}}+O\left ( \sigma ^{-4}  \right ).
   \end{equation}
   Indeed, according to scaling property of $D\left ( \phi _{1\sigma } ^2,\phi _{1\sigma }^2 \right )$, one can infer
   \begin{equation} \label{e705}
D\left ( \phi _{1\sigma } ^2,\phi _{1\sigma }^2 \right )
= \frac{ \theta ^{2}A_{\sigma}^{4}} { \left \| Q \right \| _{2}^{4}}\sigma D\left ( \varphi _{\sigma } ^2Q^{2},\varphi _{\sigma } ^2Q^{2} \right ).
  \end{equation}
 From Newton’s theorem (see \cite{Lenzmann}), we have
    \begin{equation} \label{e706}
 \left | x \right |^{ -1}\ast  Q ^{2}    \le  \frac{\left \| Q \right \|_{2}^{2}}{\left | x \right | },
   \end{equation}
   which implies that
  \begin{equation} \label{e707}
  \begin{split}
 &\left | D\left ( Q ^2,\left ( \varphi _{\sigma } ^2-1 \right ) Q^{2} \right )+D\left (\varphi _{\sigma } ^2 Q ^2,\left ( \varphi _{\sigma } ^2-1 \right ) Q^{2} \right ) \right |\\
 \le& 2D\left ( Q ^2,\left (1- \varphi _{\sigma } ^2 \right ) Q^{2} \right )= 2 \int\limits_{\mathbb{R} ^{3} }\left ( \left | x \right |^{ -1}\ast  Q ^{2}  \right )\left (1- \varphi _{\sigma } ^2 \right ) Q^{2}dx  \\
\le& 2\left \| Q \right \|_{2}^{2}\int\limits_{\mathbb{R} ^{3} }\frac{\left (1- \varphi _{\sigma } ^2 \right ) Q^{2}(x)}{\left | x \right | }dx=  2\left \| Q \right \|_{2}^{2}\int\limits_{\mathbb{R} ^{3}\setminus B_{\sigma }  }\frac{Q^{2}(x)}{\left | x \right | }dx\\
\le& C\int\limits_{\mathbb{R} ^{3}\setminus B_{\sigma }  }\left | x\right |^{-8} dx\\
\le& C\sigma^{-5},
    \end{split}
 \end{equation}
as $\sigma$ large enough. In fact, it holds
   \begin{equation} \label{e709}
    D\left ( \varphi _{\sigma } ^2Q^{2},\varphi _{\sigma } ^2Q^{2} \right )=D\left ( Q ^2,Q^{2} \right )+D\left ( Q ^2,\left ( \varphi _{\sigma } ^2-1 \right ) Q^{2} \right )+D\left (\varphi _{\sigma } ^2 Q ^2,\left ( \varphi _{\sigma } ^2-1 \right ) Q^{2} \right ),
 \end{equation}
hence we further deduce
    \begin{equation} \label{e710}
  D\left ( Q ^2,Q^{2} \right )-C\sigma^{-5}\le  D\left ( \varphi _{\sigma } ^2Q^{2},\varphi _{\sigma } ^2Q^{2} \right )  \\
\le D\left ( Q ^2,Q^{2} \right )+C\sigma^{-5}.
 \end{equation}
 Combining (\ref{e39}), (\ref{e705}) and (\ref{e710}), we thus conclude
 \begin{equation} \label{e42}
\frac{\theta ^{2}\sigma  }{\left \| Q \right \| _{2}^{4} }D\left ( Q^{2},Q^{2}  \right )-O\left ( \sigma ^{-4}  \right )
 \le D\left ( \phi _{1\sigma } ^2,\phi _{1\sigma }^2 \right )
 \le \frac{\theta ^{2}\sigma  }{\left \| Q \right \| _{2}^{4} }D\left ( Q^{2},Q^{2}  \right )+O\left ( \sigma ^{-4}  \right ),
 \end{equation}
 which shows from (\ref{e21}) that claim (\ref{e708}) is true. Note that, in a similar way, we can also obtain the estimates of $D\left ( \phi _{2\sigma } ^2,\phi _{2\sigma }^2 \right )$ and $D\left ( \phi _{1\sigma } ^2,\phi _{2\sigma }^2 \right )$  respectively that
  \begin{equation} \label{e43}
 \frac{ 2\left ( 1-\theta \right )  ^{2}\sigma} { a^{\ast}}-O\left ( \sigma ^{-4}  \right )\le D\left ( \phi _{2\sigma } ^2,\phi _{2\sigma }^2 \right )\le \frac{ 2\left ( 1-\theta \right )  ^{2}\sigma} { a^{\ast}}+O\left ( \sigma ^{-4}  \right ),
 \end{equation}
 \begin{equation} \label{e44}
 \frac{ 2\theta\left ( 1-\theta \right )  \sigma} { a^{\ast}}-O\left ( \sigma ^{-4}  \right )\le D\left ( \phi _{1\sigma } ^2,\phi _{2\sigma }^2 \right )\le \frac{ 2\theta\left ( 1-\theta \right ) \sigma} { a^{\ast}}+O\left ( \sigma ^{-4}  \right ).
 \end{equation}
Furthermore, since the function $x\to V_{i}\left ( x \right )\varphi ^{2}\left ( x-x_{0}  \right )(i=1,2)$ is bounded, by the convergence theorem, we infer
  \begin{equation} \label{e45}
  \begin{split}
 \lim_{\sigma  \to \infty } \int\limits_{\mathbb{R} ^{3} }V_{1}\left ( x \right )\phi _{1\sigma }^2dx&=\lim_{\sigma  \to \infty } \int\limits_{\mathbb{R} ^{3} }\frac{\theta A_{\sigma }^{2}   }{\left \| Q \right \| _{2}^{2} }\sigma ^{3}V_{1}\left ( x \right )\varphi ^{2}\left ( x-x_{0}  \right ) Q ^{2}\left ( \sigma \left ( x-x_{0}  \right )   \right )dx \\
&= \lim_{\sigma  \to \infty } \int\limits_{\mathbb{R} ^{3} }\frac{\theta A_{\sigma }^{2}   }{\left \| Q \right \| _{2}^{2} }V_{1}\left ( \sigma ^{-1}x+x_{0}  \right )\varphi ^{2}\left ( \sigma ^{-1}x  \right ) Q ^{2}\left ( x  \right )dx\\
&= \frac{\theta    }{\left \| Q \right \| _{2}^{2} } V_{1}\left ( x_{0}  \right )\int\limits_{\mathbb{R} ^{3} }Q^{2}\left ( x \right )dx=\theta V_{1}\left ( x_{0}  \right ),
    \end{split}
 \end{equation}
 and similarly we can also obtain
 \begin{equation} \label{e46}
\lim_{\sigma  \to \infty } \int\limits_{\mathbb{R} ^{3} }V_{2}\left ( x \right )\phi _{2\sigma }^2dx=\left ( 1-\theta \right )  V_{2}\left ( x_{0}  \right ).
 \end{equation}
 Then, it follows from (\ref{e29}), (\ref{e40})-(\ref{e708}) and (\ref{e43})-(\ref{e46}) that
   \begin{equation} \label{e47}
  \begin{split}
 E_{a _{1}, a_{2},\beta}(\phi _{1\sigma },\phi _{2\sigma })=& \sum_{i=1}^{2} \int\limits_{\mathbb{R} ^{3} }\left[\phi _{i\sigma }\sqrt{- \triangle +  m^{2} }\phi _{i\sigma }+V_{i}(x)\phi _{i\sigma }^{2}\right]dx \\
&-\frac{1}{2}\left [ a_{1}D\left ( \phi _{1\sigma } ^{2},\phi _{1\sigma } ^{2}  \right )+ a_{2}D\left ( \phi _{2\sigma } ^{2},\phi _{2\sigma } ^{2}  \right )+2\beta D\left ( \phi _{1\sigma } ^{2},\phi _{2\sigma } ^{2}  \right )  \right ] \\
\le&\sigma -\frac{ a_{1} \theta ^{2}} { a^{\ast}}\sigma-\frac{ a_{2} \left ( 1-\theta \right )  ^{2}} { a^{\ast}}\sigma -\frac{ 2\beta  \theta \left ( 1-\theta \right ) } { a^{\ast}}\sigma +O\left ( \sigma ^{-\frac{5}{2} }  \right )+2m^{2}O\left ( \sigma ^{-1}  \right )\\
 &+\int\limits_{\mathbb{R} ^{3} }V_{1}(x)\phi _{1\sigma }^{2}dx+\int\limits_{\mathbb{R} ^{3} }V_{2}(x)\phi _{2\sigma }^{2}dx\\
 =&\frac{\sigma}{a^{\ast }} \left [ a^{\ast } - a_{1} \theta ^{2}-a_{2} \left ( 1-\theta \right )  ^{2} -2\beta  \theta \left ( 1-\theta \right )  \right ] +O\left ( \sigma ^{-\frac{5}{2} }  \right )+2m^{2}O\left ( \sigma ^{-1}  \right )\\
 &+\theta V_{1}\left ( x_{0}  \right )+\left ( 1-\theta  \right )V_{2}\left ( x_{0}  \right ).
    \end{split}
 \end{equation}
Next, we finish the proof of (ii) by considering the following three cases.\\
 \noindent{\bf Case 1:} $a_{1}> a^{\ast }$. One can obtain from (\ref{e47}) by taking $\theta =1$ and $\forall x_{0}\in \mathbb{R} ^{3}$ that
$$ E_{a _{1}, a_{2},\beta}(\phi _{1\sigma },\phi _{2\sigma })\le \frac{\sigma}{a^{\ast }} \left ( a^{\ast } - a_{1} \right )+V_{1}\left ( x_{0}  \right )+O\left ( \sigma ^{-\frac{5}{2} }  \right )+2m^{2}O\left ( \sigma ^{-1}  \right ) \to -\infty \quad {\rm as}\ \sigma \to \infty ,$$
which implies $e\left ( a _{1}, a_{2},\beta \right )=-\infty$ in this case. Therefore, if $a_{1}> a^{\ast }$, we deduce that there is no minimizer of (\ref{e1}).\\
\noindent{\bf Case 2:} $a_{2}> a^{\ast }$. By taking $\theta =0$ and $\forall x_{0}\in \mathbb{R} ^{3}$ in (\ref{e47}), one can infer
$$E_{a _{1}, a_{2},\beta}(\phi _{1\sigma },\phi _{2\sigma })\le \frac{\sigma}{a^{\ast }} \left ( a^{\ast } - a_{2} \right )+V_{2}\left ( x_{0}  \right )+O\left ( \sigma ^{-\frac{5}{2} }  \right )+2m^{2}O\left ( \sigma ^{-1}  \right ) \to -\infty \quad {\rm as}\ \sigma \to \infty, $$
which indicates $e\left ( a _{1}, a_{2},\beta \right )=-\infty$. Hence, if $a_{2}> a^{\ast }$, we conclude that there is no minimizer of (\ref{e1}).\\
\noindent{\bf Case 3:} $\beta > \beta ^{\ast }$. Without losing generality, one can suppose $ 0< a_{1}, a_{2}< a^{\ast }$ in this case. From (\ref{e48}), one can easily get that
 $$a^{\ast }-a _{1}\gamma ^{2}-a_{2} \left ( 1-\gamma \right )  ^{2}=2 \beta ^{\ast }\gamma \left ( 1-\gamma  \right ) .$$
 Let $\theta = \gamma $ and $\forall x_{0}\in \mathbb{R} ^{3}$ in (\ref{e47}), we therefore deduce that
 \begin{equation} \label{e49}
  \begin{split}
 E_{a _{1}, a_{2},\beta}(\phi _{1\sigma },\phi _{2\sigma })\le& \frac{2\sigma}{a^{\ast }} \gamma \left ( 1-\gamma  \right ) \left ( \beta ^{\ast } - \beta  \right )+\gamma V_{1}\left ( x_{0}  \right )+\left ( 1-\gamma  \right ) V_{2}\left ( x_{0}  \right )\\
 &+O\left ( \sigma ^{-\frac{5}{2} }  \right )+2m^{2}O\left ( \sigma ^{-1}  \right ) \to -\infty \quad {\rm as}\ \sigma \to \infty,
    \end{split}
 \end{equation}
 which shows $e\left ( a _{1}, a_{2},\beta \right )=-\infty$ as well. We thus derive, if $\beta > \beta ^{\ast }$, that there is no minimizer of (\ref{e1}).
\begin{enumerate}[(\romannumeral3)]
\item Finally, we prove that there is no minimizer of (\ref{e1}) in the case  $a_{1}= a^{\ast }$ or $a_{2}= a^{\ast }$ or $\beta= \beta ^{\ast }$.
 \end{enumerate}
 According to $\left ( {\rm ii} \right ) $, one can assume that $0< a_{1}, a_{2}\le a^{\ast }$ and $0< \beta \le \beta ^{\ast }$. From the assumption $\left ( \Omega  \right )$, we know that
 \begin{equation} \label{e50}
 \mathcal{V}=\left \{ x\in \mathbb{R} ^{3}:V_{1}\left ( x \right )=V_{2}\left ( x \right )=0      \right \}\ne \emptyset .
  \end{equation}
 Therefore, we can choose $x_{0}\in  \mathcal{V} $ such that $V_{1}\left (x_{0} \right )=V_{2}\left ( x_{0} \right )=0 $.
 It then follows from (\ref{e51}) and (\ref{e47}) that
  \begin{equation} \label{e52}
    \begin{split}
  0\le e\left ( a_{1}, a_{2},\beta \right )\le E_{a _{1}, a_{2},\beta}(\phi _{1\sigma },\phi _{2\sigma })\le
&\frac{\sigma}{a^{\ast }} \left [ a^{\ast } - a_{1} \theta ^{2}-a_{2} \left ( 1-\theta \right )  ^{2} -2\beta  \theta \left ( 1-\theta \right )  \right ] \\
&+O\left ( \sigma ^{-\frac{5}{2} }  \right )+2m^{2}O\left ( \sigma ^{-1}  \right ).
  \end{split}
  \end{equation}
  We next discuss the following three cases similarly as part $\left ( {\rm ii} \right ) $.\\
\noindent{\bf Case 1:}  $a_{1}= a^{\ast }$, $0< a_{2}\le a^{\ast }$ and $0< \beta \le \beta ^{\ast }$. Let $\theta =1$ in  (\ref{e52}) and $\sigma \to \infty $, we then obtain
  \begin{equation} \label{e53}
 0\le e\left ( a^{\ast }, a_{2},\beta \right )\le\lim_{\sigma  \to \infty }  E_{a ^{\ast }, a_{2},\beta}(\phi _{1\sigma },\phi _{2\sigma })=0.
   \end{equation}
   Next, we argue by contradiction. Suppose $e\left ( a^{\ast }, a_{2},\beta \right )$ has a minimizer $\left ( u_{10},u_{20} \right )$. From (\ref{e51}) and (\ref{e53}), we infer that
 \begin{equation} \label{e54}
   \int\limits_{\mathbb{R} ^{3} }\left(\left | (-\triangle )^{\frac{1}{4} }u_{10}  \right |^{2}+\left | (-\triangle )^{\frac{1}{4} }u_{20}  \right |^{2}\right )dx=\frac{a^{\ast }}{2}D\left ( u_{10}^{2}+u_{20}^{2},u_{10}^{2}+u_{20}^{2} \right ),
 \end{equation}
 \begin{equation} \label{e55}
 \int\limits_{\mathbb{R} ^{3} } V_{1}(x)u_{10}^{2}dx=\int\limits_{\mathbb{R} ^{3} }V_{2}(x)u_{20}^{2} dx=0  .
  \end{equation}
 One can deduce from Lemma \ref{le2} and  the equality  (\ref{e54}) that
 \begin{equation} \label{e56}
\left ( u_{10}(x), u_{20}(x) \right )=\left ( \tau sin\theta  Q\left (\alpha x+\eta  \right ) ,\tau cos\theta Q(\alpha x+\eta ) \right )
	\end{equation}
 with some $\tau>0, \alpha>0, \theta \in \left [ 0,2\pi  \right )$ and $\eta \in \mathbb{R} ^{3}$. However, (\ref{e55}) shows that there is a compact support of $(u_{10},u_{20})$, which contradicts with (\ref{e56}). Therefore, in this case, there is no minimizer of $e\left ( a^{\ast }, a_{2},\beta \right )$.\\
 \noindent{\bf Case 2:} $a_{2}= a^{\ast }$, $0< a_{1}\le a^{\ast }$ and $0< \beta \le \beta ^{\ast }$. Let $\theta =0$ in  (\ref{e52}) and $\sigma \to \infty $, we then have
 \begin{equation} \label{e57}
 0\le e\left ( a_{1 }, a^{\ast},\beta \right )\le\lim_{\sigma  \to \infty }  E_{a_{1}, a^{\ast},\beta}(\phi _{1\sigma },\phi _{2\sigma })=0,
   \end{equation}
   which shows that there is no minimizer of $e\left ( a_{1 }, a^{\ast},\beta \right )$ by arguing similarly as above.\\
   \noindent{\bf Case 3:} $\beta= \beta^{\ast }$ and $0< a_{1}, a_{2}\le a^{\ast }$. Let $\theta =\gamma $ in  (\ref{e52}) and $\sigma \to \infty $, we then get
 \begin{equation} \label{e58}
 0\le e\left ( a_{1 }, a_{2 },\beta^{\ast} \right )\le\lim_{\sigma  \to \infty }  E_{a _{1 }, a_{2},\beta^{\ast }}(\phi _{1\sigma },\phi _{2\sigma })=0,
   \end{equation}
   which implies that there is no minimizer of $e\left ( a_{1 }, a_{2 },\beta^{\ast} \right )$ by arguing similarly as above.\\

Furthermore, in the case  $0< a_{1}, a_{2}< a^{\ast } $ and $0< \beta < \beta ^{\ast } $, let $\theta =\gamma $  and $\sigma = \left ( \beta ^{\ast }-\beta \right )   ^{-\frac{1}{2} }$ in (\ref{e52}), one can infer
 \begin{equation} \label{e59}
 0\le e\left ( a_{1}, a_{2},\beta \right )\le \frac{2\sigma}{a^{\ast }} \gamma \left ( 1-\gamma  \right ) \left ( \beta ^{\ast } - \beta  \right )
 +O\left ( \sigma ^{-\frac{5}{2} }  \right )+2m^{2}O\left ( \sigma ^{-1}  \right ) \to 0 \quad {\rm as}\ \beta \to \beta^{\ast}.
    \end{equation}
Combining this and (\ref{e58}), we thus deduce that
     \begin{equation} \label{e67}
    \lim_{\beta  \to \beta ^{\ast } }  e\left ( a_{1}, a_{2},\beta \right )=e\left ( a_{1}, a_{2},\beta^{\ast } \right )=0, \quad  0< a_{1}, a_{2}< a^{\ast },
    \end{equation}
    which completes the proof of Theorem \ref{T1}.\\

\section{General concentration behavior}
In this section, we are devoted to the concentration behavior of the minimizers for (\ref{e1}) in the case $0< a_{1}, a_{2}< a^{\ast } $ and  $0<\beta < \beta ^{\ast } $ under general trapping potentials. For this purpose, we start with the following two lemmas.

\begin{lemma} \label{le3}
Suppose $\left ( \Omega  \right )$ holds and let $\left ( u_{1\beta }, u_{2\beta }  \right )$ be a nonnegative minimizer of $e\left ( a_{1}, a_{2},\beta \right )$ for $0< a_{1}, a_{2}< a^{\ast } $ and $0<\beta < \beta ^{\ast } $, then it holds
\begin{enumerate}[(\romannumeral1)]
\item  $\left ( u_{1\beta }, u_{2\beta }  \right )$ blows up as $\beta  \to \beta ^{\ast } $, i.e., for $i=1,2$,
 \begin{equation} \label{e60}
 \lim_{\beta  \to \beta ^{\ast } } \int\limits_{\mathbb{R} ^{3} }\left | (-\triangle )^{\frac{1}{4} }u_{i\beta }   \right | ^{2}dx=\infty \ {\rm and} \ \lim_{\beta  \to \beta ^{\ast } } \iint\limits_{\mathbb{R} ^{3}\times \mathbb{R} ^{3}}\frac{u_{i\beta } ^{2}(x)u_{i\beta } ^{2}(y) }{\left | x-y \right | }dxdy=\infty.
  \end{equation}
\item $\left ( u_{1\beta }, u_{2\beta }  \right )$ also satisfies
\begin{equation} \label{e61}
\lim_{\beta  \to \beta ^{\ast } } \int\limits_{\mathbb{R} ^{3} }V_{1}(x)u_{1\beta }^{2}= \lim_{\beta  \to \beta ^{\ast } } \int\limits_{\mathbb{R} ^{3} }V_{2}(x)u_{2\beta }^{2}=0,
\end{equation}
\begin{equation} \label{e62}
\lim_{\beta  \to \beta ^{\ast } }\left [ \left ( \beta ^{\ast }-\beta   \right )D\left ( u_{1\beta }^{2},u_{2\beta }^{2} \right )   \right ]=0,
\end{equation}
\begin{equation} \label{e63}
\lim_{\beta  \to \beta ^{\ast } } \left [ \sqrt{\left ( a^{\ast }-a_{1}   \right ) D\left ( u_{1\beta }^{2},u_{1\beta }^{2} \right )}-\sqrt{\left ( a^{\ast }-a_{2}   \right ) D\left ( u_{2\beta }^{2},u_{2\beta }^{2} \right )} \right ]^{2}=0,
\end{equation}
\begin{equation} \label{e64}
\lim_{\beta  \to \beta ^{\ast } } \left [ \sqrt{ D\left ( u_{1\beta }^{2},u_{1\beta }^{2} \right )D\left ( u_{2\beta }^{2},u_{2\beta }^{2} \right )}-D\left ( u_{1\beta }^{2},u_{2\beta }^{2} \right )  \right ]=0,
\end{equation}
as well as
\begin{equation} \label{e65}
\lim_{\beta  \to \beta ^{\ast } }  \frac{\displaystyle \int_{\mathbb{R} ^{3} }  \left ( \left | (-\triangle )^{\frac{1}{4} }u_{1\beta }   \right | ^{2}+\left | (-\triangle )^{\frac{1}{4} }u_{2\beta }   \right | ^{2} \right ) dx}{ D\left ( u_{1\beta }^{2}+u_{2\beta }^{2},u_{1\beta }^{2}+u_{2\beta }^{2} \right )}=\frac{a^{\ast }}{2},
\end{equation}
\begin{equation} \label{e66}
\lim_{\beta  \to \beta ^{\ast } }  \frac{D\left ( u_{1\beta }^{2},u_{1\beta }^{2} \right )}{ D\left ( u_{2\beta }^{2},u_{2\beta }^{2} \right )}=\frac{a^{\ast}-a_{2}  }{a^{\ast}-a_{1}  }.
\end{equation}
\end{enumerate}
\end{lemma}
 \begin{proof} According to (\ref{e67}), one can immediately deduce from (\ref{e51}), Lemma \ref{b3} and refined Gagliardo-Nirenberg inequality (\ref{e27}) that (\ref{e61})-(\ref{e64}) and the following hold,
\begin{equation} \label{e68}
\lim_{\beta  \to \beta ^{\ast } } \int\limits_{\mathbb{R} ^{3} }\left(\left | (-\triangle )^{\frac{1}{4} }u_{1\beta }   \right | ^{2}+\left | (-\triangle )^{\frac{1}{4} }u_{2\beta }   \right | ^{2}\right )dx- \frac{a^{\ast }}{2}D\left ( u_{1\beta }^{2}+u_{2\beta }^{2},u_{1\beta }^{2}+u_{2\beta }^{2} \right )=0.
\end{equation}
Note that
\begin{equation} \label{e69}
\lim_{\beta  \to \beta ^{\ast } } \int\limits_{\mathbb{R} ^{3} }\left(\left | (-\triangle )^{\frac{1}{4} }u_{1\beta }   \right | ^{2}+\left | (-\triangle )^{\frac{1}{4} }u_{2\beta }   \right | ^{2}\right )dx= \infty.
\end{equation}
In fact, suppose (\ref{e69}) does not hold, then we can infer from (\ref{e61}) that there exists a sequence $\left \{ \beta _{k}  \right \}$ satisfying $\beta _{k} \to \beta ^{\ast }$ as $k\to \infty $, such that $\left \{ \left ( u_{1 \beta _{k} }  , u_{2 \beta _{k} }  \right )  \right \} $ is bounded uniformly in $\mathcal{H} _{1}\times \mathcal{H} _{2} $. Therefore, from Lemma \ref{b1}, we conclude that there exists a subsequence of $\left \{ \beta _{k}  \right \}$, still denoted by $\left \{ \beta _{k}  \right \}$, and $\left ( u_{10},u_{20}   \right )\in  \mathcal{H} _{1}\times \mathcal{H} _{2} $ such that
\begin{equation} \label{e70}
\begin{cases}
\left ( u_{1\beta _{k}},u_{2\beta _{k}}  \right ) \rightharpoonup  \left ( u_{10},u_{20}  \right ) \quad {\rm weakly \ in}\  \mathcal{H} _{1}\times \mathcal{H} _{2},\\
\left ( u_{1\beta _{k}},u_{2\beta _{k}}  \right ) \to \left ( u_{10},u_{20}  \right ) \quad {\rm strongly \ in}\ L^{q}\left ( \mathbb{R} ^{3}  \right )\times L^{q}\left ( \mathbb{R} ^{3}  \right ), \  2\le q< 3.
\end{cases}
\end{equation}
Then, it follows from Lemma \ref{b4} and the weak lower semi-continuity of $\left (  \sqrt{-\triangle +m^{2} }u,u \right )$ that
\begin{equation} \label{e71}
0=e\left ( a_{1}, a_{2}, \beta^{\ast }   \right )\le E_{a_{1}, a_{2}, \beta^{\ast }}\left ( u_{10}, u_{20} \right )\le \lim_{k \to \infty } E_{a_{1}, a_{2}, \beta_{k}}\left ( u_{1\beta_{k} }, u_{2\beta_{k}} \right )=\lim_{k \to \infty }e\left ( a_{1}, a_{2}, \beta_{k}   \right )=0.
\end{equation}
It suggests that $\left ( u_{10}, u_{20} \right )$ is a minimizer of $e\left ( a_{1}, a_{2}, \beta^{\ast }   \right )$, which contradicts Theorem \ref{T1} and hence (\ref{e69}) holds.
 Combining (\ref{e68}) and (\ref{e69}), (\ref{e65}) thus follows. Moreover, from (\ref{e64}),(\ref{e65}) and (\ref{e69}), we further conclude that
 \begin{equation} \label{e105}
  \lim_{\beta  \to \beta ^{\ast } }  D\left ( u_{1\beta }^{2}+u_{2\beta }^{2},u_{1\beta }^{2}+u_{2\beta }^{2} \right )= \lim_{\beta  \to \beta ^{\ast } } \left (\sqrt{D\left ( u_{1\beta }^{2},u_{1\beta }^{2} \right )} +\sqrt{D\left ( u_{2\beta }^{2},u_{2\beta }^{2} \right )}  \right )^{2}=\infty.
 \end{equation}
It then follows from (\ref{e63}) that for $i=1,2$,
 \begin{equation} \label{e106}
 \lim_{\beta  \to \beta ^{\ast } } \iint\limits_{\mathbb{R} ^{3}\times \mathbb{R} ^{3}}\frac{u_{i\beta } ^{2}(x)u_{i\beta } ^{2}(y) }{\left | x-y \right | }dxdy=\infty.
  \end{equation}
Therefore, we can deduce from the Gagliardo-Nirenberg inequality  (\ref{e17}) that (\ref{e60}) holds.
  Finally,  (\ref{e60}) and (\ref{e63}) show that (\ref{e66}) is true, which completes the proof.\\

For the nonnegative minimizer $\left ( u_{1\beta },u_{2\beta }  \right )$ of (\ref{e1}), define
 \begin{equation} \label{e72}
 \varepsilon _{\beta }:= \left ( \int\limits_{\mathbb{R} ^{3} }\left ( \left | (-\triangle )^{\frac{1}{4} }u_{1\beta }   \right | ^{2}+\left | (-\triangle )^{\frac{1}{4} }u_{2\beta }   \right | ^{2} \right ) dx \right )^{-1}>0,
 \end{equation}
  which satisfies $\varepsilon _{\beta }\to 0$ as $\beta \to \beta ^{\ast }$ according to (\ref{e69}). Next, we introduce another lemma that will play a fundamental role in the proof of Theorem \ref{T2}.
\end{proof}
\begin{lemma}\label{b3.2} Suppose $\left ( \Omega  \right )$ holds and let $\left ( u_{1\beta }, u_{2\beta }  \right )$ be a nonnegative minimizer of $e\left ( a_{1}, a_{2},\beta \right )$ for $0< a_{1}, a_{2}< a^{\ast } $ and $0<\beta < \beta ^{\ast } $. Then for any sequence $\beta _{k} \to \beta ^{\ast }$ as $k\to \infty $, we obtain
\begin{enumerate}[(\romannumeral1)]
\item There exist a sequence $\left \{ \eta_{\beta _{k} }  \right \}\subset \mathbb{R} ^{3}$ and constants $r_{0}, \rho>0$ such that
 \begin{equation} \label{e73}
w_{i\beta _{k} }\left ( x \right ):=\varepsilon _{\beta _{k} }^{\frac{3}{2} }u_{i\beta _{k} }\left ( \varepsilon _{\beta _{k} }x+\varepsilon _{\beta _{k} }\eta _{\beta _{k} } \right )
 \end{equation}
 satisfies
  \begin{equation} \label{e74}
 \liminf _{k\to \infty }\int\limits_{B_{r_{0} }(0) }\left | w_{i\beta _{k} }\left ( x \right )  \right |  ^{2}dx\ge \rho > 0, \quad i=1,2.
  \end{equation}
  Furthermore,
\begin{equation} \label{e75}
\int\limits_{\mathbb{R} ^{3} }\left(\left | (-\triangle )^{\frac{1}{4} }w_{1\beta _{k} }   \right | ^{2}+\left | (-\triangle )^{\frac{1}{4} }w_{2\beta _{k} }   \right | ^{2}\right)dx=
\int\limits_{\mathbb{R} ^{3} }\left(\left | w_{1\beta _{k} }   \right | ^{2}+\left |w_{2\beta _{k} }   \right | ^{2}\right)dx=1,
  \end{equation}
\begin{equation} \label{e76}
\lim_{k  \to \infty  }\left [ \frac{\beta ^{\ast }-\beta_{k} }{\varepsilon _{\beta _{k}} } D\left ( w_{1\beta _{k} }^{2},w_{2\beta _{k} }^{2} \right )   \right ]=0,
  \end{equation}
  \begin{equation} \label{e77}
  \lim_{k  \to \infty  } D\left ( w_{1\beta _{k} }^{2}+w_{2\beta _{k} }^{2},w_{1\beta _{k} }^{2}+w_{2\beta _{k} }^{2} \right ) =\frac{2}{a^{\ast } } ,
    \end{equation}
\begin{equation} \label{e78}
\begin{split}
\lim_{k  \to \infty  } D\left ( w_{1\beta _{k} }^{2},w_{1\beta _{k} }^{2} \right ) &=\frac{2\gamma ^{2} }{a^{\ast } }, \
\lim_{k  \to \infty  } D\left ( w_{2\beta _{k} }^{2},w_{2\beta _{k} }^{2} \right ) =\frac{2\left ( 1-\gamma  \right ) ^{2} }{a^{\ast } },\\
&\lim_{k  \to \infty  } D\left ( w_{1\beta _{k} }^{2},w_{2\beta _{k} }^{2} \right ) =\frac{2\gamma \left ( 1-\gamma  \right )  }{a^{\ast } }.
\end{split}
 \end{equation}
\item There exists a subsequence of $\left \{ \beta _{k}  \right \} $, still denoted by $\left \{ \beta _{k}  \right \}$, such that
\begin{equation} \label{e79}
 x_{\beta _{k}}:=\varepsilon _{\beta _{k}}\eta_{\beta _{k}}\to x_{0} \quad {\rm as}\ k\to \infty
  \end{equation}
for some $x_{0}\in \mathcal{V}$.
\item There exists a subsequence of $\left \{ \beta _{k}  \right \} $, still denoted by $\left \{ \beta _{k}  \right \}$, and $\eta_{0}\in \mathbb{R} ^{3} $ such that
\begin{equation} \label{e80}
\begin{cases}
\lim_{k  \to \infty  }w_{1\beta _{k} }\left ( x \right )= \frac{\sqrt{\gamma  }}{\left \| Q \right \| _{2} }Q\left (  x-\eta_{0}   \right ),\\
\lim_{k  \to \infty  }w_{2\beta _{k} }\left ( x \right )= \frac{\sqrt{1-\gamma  }}{\left \| Q \right \| _{2} }Q\left (  x-\eta_{0}   \right ),
\end{cases}
\end{equation}
strongly in $H^{\frac{1}{2} }\left ( \mathbb{R} ^{3}  \right ) $.

\end{enumerate}
\end{lemma}
\begin{proof}
\begin{enumerate}[(\romannumeral1)]
\item Denote
\begin{equation} \label{e81}
v_{i\beta _{k} }\left ( x \right ):=\varepsilon _{\beta _{k} }^{\frac{3}{2} }u_{i\beta _{k} }\left ( \varepsilon _{\beta _{k} }x\right ), \quad i=1,2.
\end{equation}
By a simple calculation, one can obtain from (\ref{e72}) that
\begin{equation} \label{e82}
\begin{split}
&\displaystyle \int\limits_{\mathbb{R} ^{3} }\left (\left | (-\triangle )^{\frac{1}{4} }v_{1\beta _{k} }   \right | ^{2}+\left | (-\triangle )^{\frac{1}{4} }v_{2\beta _{k} }   \right | ^{2}\right )dx\\
=&\varepsilon _{\beta _{k} }\int\limits_{\mathbb{R} ^{3} }\left (\left | (-\triangle )^{\frac{1}{4} }u_{1\beta _{k} }   \right | ^{2}+\left | (-\triangle )^{\frac{1}{4} }u_{2\beta _{k} }   \right | ^{2}\right )dx\\
=&1\\
=&\int\limits_{\mathbb{R} ^{3} }\left (\left | u_{1\beta _{k} }   \right | ^{2}+\left | u_{2\beta _{k} }   \right | ^{2}\right )dx\\
=&\int\limits_{\mathbb{R} ^{3} }\left (\left | v_{1\beta _{k} }   \right | ^{2}+\left | v_{2\beta _{k} }   \right | ^{2}\right )dx.
\end{split}
\end{equation}
Besides, it comes from $D\left ( v_{1\beta _{k} }^{2},v_{2\beta _{k} }^{2} \right )=\varepsilon _{\beta _{k} }D\left ( u_{1\beta _{k} }^{2},u_{2\beta _{k} }^{2} \right )$ and  (\ref{e62}) that
\begin{equation} \label{e83}
\lim_{k  \to \infty  } \left [ \frac{\beta ^{\ast }-\beta_{k} }{\varepsilon _{\beta _{k}} } D\left ( v_{1\beta _{k} }^{2},v_{2\beta _{k} }^{2} \right )   \right ]=\lim_{k  \to \infty  } \left [ \left ( \beta ^{\ast }-\beta_{k} \right )   D\left ( u_{1\beta _{k} }^{2},u_{2\beta _{k} }^{2} \right ) \right ] =0.
\end{equation}
Similarly from (\ref{e63}) and (\ref{e64}), we can obtain
$$
\lim_{k  \to \infty  } \left [ \sqrt{\left ( a^{\ast }-a_{1}   \right ) D\left ( v_{1\beta _{k} }^{2},v_{1\beta _{k} }^{2} \right )}-\sqrt{\left ( a^{\ast }-a_{2}   \right ) D\left ( v_{2\beta _{k} }^{2},v_{2\beta _{k} }^{2} \right )} \right ]^{2}=0,
$$
and
$$\lim_{k  \to \infty  } \left [ \sqrt{ D\left ( v_{1\beta _{k} }^{2},v_{1\beta _{k} }^{2} \right )D\left ( v_{2\beta _{k} }^{2},v_{2\beta _{k} }^{2} \right )}-D\left ( v_{1\beta _{k} }^{2},v_{2\beta _{k} }^{2} \right )  \right ]=0.$$
Moreover, from (\ref{e65}) and (\ref{e82}), one can easily conclude that
\begin{equation} \label{e84}
\begin{split}
&\lim_{k  \to \infty  } D\left ( v_{1\beta _{k} }^{2}+v_{2\beta _{k} }^{2},v_{1\beta _{k} }^{2}+v_{2\beta _{k} }^{2} \right )\\
=&\lim_{k  \to \infty  } \frac{D\left ( v_{1\beta _{k} }^{2}+v_{2\beta _{k} }^{2},v_{1\beta _{k} }^{2}+v_{2\beta _{k} }^{2} \right ) }{\displaystyle \int\limits_{\mathbb{R} ^{3} }\left (\left | (-\triangle )^{\frac{1}{4} }v_{1\beta _{k} }   \right | ^{2}+\left | (-\triangle )^{\frac{1}{4} }v_{2\beta _{k} }   \right | ^{2}\right )dx} \\
=& \lim_{k  \to \infty  } \frac{D\left ( u_{1\beta _{k} }^{2}+u_{2\beta _{k} }^{2},u_{1\beta _{k} }^{2}+u_{2\beta _{k} }^{2} \right ) }{\displaystyle \int\limits_{\mathbb{R} ^{3} }\left (\left | (-\triangle )^{\frac{1}{4} }u_{1\beta _{k} }   \right | ^{2}+\left | (-\triangle )^{\frac{1}{4} }u_{2\beta _{k} }   \right | ^{2}\right )dx}\\
=&\frac{2}{a^{\ast }} ,
\end{split}
\end{equation}
which implies that
\begin{equation} \label{e108}
\begin{split}
\frac{2}{a^{\ast }}&=\lim_{k  \to \infty  } D\left ( v_{1\beta _{k} }^{2}+v_{2\beta _{k} }^{2},v_{1\beta _{k} }^{2}+v_{2\beta _{k} }^{2} \right )\\
&=\lim_{k \to \infty } \left (\sqrt{D\left ( v_{1\beta _{k} }^{2},v_{1\beta _{k} }^{2} \right )} +\sqrt{D\left ( v_{2\beta _{k} }^{2},v_{2\beta _{k} }^{2} \right )}  \right )^{2}\\
&=\frac{1}{\gamma^{2} }\lim_{k \to \infty }D\left ( v_{1\beta _{k} }^{2},v_{1\beta _{k} }^{2} \right )\\
&=\frac{1}{\left ( 1-\gamma \right ) ^{2} }\lim_{k \to \infty }D\left ( v_{2\beta _{k} }^{2},v_{2\beta _{k} }^{2} \right ), \notag
\end{split}
\end{equation}
where $\gamma $ is defined by (\ref{e48}). Therefore, we have
\begin{equation} \label{e85}
\begin{split}
\lim_{k  \to \infty  } D\left ( v_{1\beta _{k} }^{2},v_{1\beta _{k} }^{2} \right ) &=\frac{2\gamma ^{2} }{a^{\ast } }, \
\lim_{k  \to \infty  } D\left ( v_{2\beta _{k} }^{2},v_{2\beta _{k} }^{2} \right ) =\frac{2\left ( 1-\gamma  \right ) ^{2} }{a^{\ast } },\\
&\lim_{k  \to \infty  } D\left ( v_{1\beta _{k} }^{2},v_{2\beta _{k} }^{2} \right ) =\frac{2\gamma \left ( 1-\gamma  \right )  }{a^{\ast } }.
\end{split}
\end{equation}
Note from (\ref{e82}) that $\left \{ v_{i\beta _{k}}  \right \} $ is bounded in $H^{\frac{1}{2} }\left ( \mathbb{R} ^{3}  \right ) $ for $i=1,2$. Next, we claim that there exist a sequence $\left \{ \eta_{\beta _{k}}  \right \}\subset \mathbb{R} ^{3} $ and $r_{1},\rho _{1}> 0 $ such that
 \begin{equation} \label{e86}
 \liminf _{k\to \infty }\int\limits_{B_{r_{1} }(\eta_{\beta _{k}}) }\left | v_{1\beta _{k} }\left ( x \right )  \right |  ^{2}dx\ge \rho_{1} > 0.
 \end{equation}
 Argue by contradiction that (\ref{e86}) is false. Then, for any $r> 0$, there exists a subsequence of $\left \{ v_{1\beta _{k}}  \right \} $, still denoted by $\left \{ v_{1\beta _{k}}  \right \} $, such that
 $$\lim  _{k\to \infty }\sup_{y\in \mathbb{R} ^{3} } \int\limits_{B_{r }(y) }\left | v_{1\beta _{k} }\left ( x \right )  \right |  ^{2}dx= 0,$$
 from which we can deduce that $  v_{1\beta _{k} }\to 0$ in $L^{q}\left ( \mathbb{R} ^{3}  \right )  $ for all $2< q< 3$ by Lions' vanishing lemma (see \cite{Secchi} Lemma 2.4). From (\ref{e222}), we then infer that $D\left ( v_{1\beta _{k} }^{2},v_{1\beta _{k} }^{2} \right )\to 0 $ as $k\to \infty $, which however contradicts (\ref{e85}). Therefore, (\ref{e86}) follows. From (\ref{e73}), we thus deduce
  \begin{equation} \label{e87}
   \liminf _{k\to \infty }\int\limits_{B_{r_{1} }(0) }\left | w_{1\beta _{k} }\left ( x \right )  \right |  ^{2}dx\ge \rho_{1} > 0.
 \end{equation}
Comparing (\ref{e73}) with (\ref{e81}), it follows from (\ref{e82})-(\ref{e85}) that (\ref{e75})-(\ref{e78}) hold.
Furthermore, from (\ref{e75}), one can assume that  $w_{i\beta _{k} }\rightharpoonup w_{i0}\ge 0 $ in $H^{\frac{1}{2} }\left (\mathbb{ R} ^{3}  \right ) $ for some $w_{i0}\in H^{\frac{1}{2} }\left (\mathbb{ R} ^{3}  \right ) $, where $i=1,2$. Moreover, from (\ref{e87}), one can easily see that $w_{10}\ne 0$.

To complete the proof, we only need to obtain the similar result as (\ref{e87}) for $w_{2\beta _{k} }$ with the same sequence $\eta_{\beta _{k}}$. Next, we first claim that $\left \{ x_{\beta _{k} } \right \}$ { defined in (\ref{e79}) is bounded} in $\mathbb{ R} ^{3} $. In fact, if it does not hold, there exists a subsequence of $\left \{ x_{\beta _{k} } \right \}$, still denoted by $\left \{ x_{\beta _{k} } \right \}$, such that $\left | x_{\beta _{k} } \right |\to \infty $ as $k\to \infty$. From (\ref{e61}), we infer
  \begin{equation} \label{e88}
  \sum_{i=1}^{2} \int\limits_{\mathbb{R} ^{3} }V_{i}(\varepsilon _{\beta _{k} } x+x_{\beta _{k} })\left | w_{i\beta _{k} }(x)   \right | ^{2}dx=
\sum_{i=1}^{2} \int\limits_{\mathbb{R} ^{3} }V_{i}(x)\left | u_{i\beta _{k} }(x)   \right | ^{2}dx\to 0 \quad {\rm as} \ k\to \infty .
 \end{equation}
 Since $V_{1}(x)\to \infty $ as $\left | x \right |\to \infty $, we can derive from Fatou's Lemma and (\ref{e87}) that  for $\forall C>0$,
 $$ \liminf _{k\to \infty } \sum_{i=1}^{2} \int\limits_{\mathbb{R} ^{3} }V_{i}(\varepsilon _{\beta _{k} } x+x_{\beta _{k} })\left | w_{i\beta _{k} }(x)   \right | ^{2}dx\ge
C\liminf _{k\to \infty } \int\limits_{B_{r_{1} }\left ( 0 \right )   }\left | w_{1\beta _{k} }(x)   \right | ^{2}dx\ge C\rho _{1}> 0,  $$
which contradicts (\ref{e88}) and therefore the claim follows. Then, we can infer that, going if necessary to a subsequence, there exists some point $x_{0}\in \mathbb{R} ^{3}$ such that
   \begin{equation} \label{e89}
   x_{\beta _{k} }\to x_{0} \quad {\rm as}\ k\to \infty .
    \end{equation}
Applying the translation and scaling properties (\ref{e700}) and (\ref{e701}) of pseudo-relativistic operator to $w_{i\beta _{k} }$, one can deduce from (\ref{e73}) that for $i=1,2$,
      \begin{equation} \label{e90}
          \sqrt{- \triangle +  \varepsilon _{\beta _{k} }^{2} m^{2} }w_{i\beta _{k} }\left (x \right )=\varepsilon _{\beta _{k} }^{\frac{5}{2} }\sqrt{- \triangle +  m^{2} } u_{i\beta _{k} }\left ( \varepsilon _{\beta _{k} }x+x _{\beta _{k} } \right ) .
    \end{equation}
 Moreover, based on the definition of convolution, we find
\begin{equation} \label{e112}
\begin{split}
&\left ( \left ( \left | x \right | ^{-1}\ast w_{2\beta _{k} } ^{2}  \right )  w_{1\beta _{k} } \right )\left (x \right )\\
=&\int\limits_{\mathbb{R} ^{3} }\frac{w_{2\beta_{k} } ^{2}(y) }{\left | x-y \right | }dy\cdot w_{1\beta_{k} }\left ( x \right ) \\
=&\int\limits_{\mathbb{R} ^{3} }\frac{\varepsilon _{\beta _{k} }^{3 }u_{2\beta_{k} } ^{2}\left ( \varepsilon _{\beta _{k} }y+x _{\beta _{k} }  \right ) }{\left | x-y \right | }dy\cdot \varepsilon _{\beta _{k} }^{\frac{3}{2} } u_{1\beta_{k} }\left ( \varepsilon _{\beta _{k} }x+x _{\beta _{k} }  \right )\\
=&\varepsilon _{\beta _{k} }^{\frac{3}{2} }\int\limits_{\mathbb{R} ^{3} }\frac{u_{2\beta_{k} } ^{2}\left ( \zeta  \right ) }{\left | x-\frac{\zeta -x _{\beta _{k} } }{\varepsilon  _{\beta _{k} } }  \right | }d\zeta \cdot  u_{1\beta_{k} }\left ( \varepsilon _{\beta _{k} }x+x _{\beta _{k} }  \right )\\
=&\varepsilon _{\beta _{k} }^{\frac{5}{2} }\int\limits_{\mathbb{R} ^{3} }\frac{u_{2\beta_{k} } ^{2}\left ( \zeta  \right ) }{\left | \varepsilon _{\beta _{k} }x+x _{\beta _{k} } -\zeta  \right | }d\zeta \cdot  u_{1\beta_{k} }\left ( \varepsilon _{\beta _{k} }x+x _{\beta _{k} }  \right )\\
=&\varepsilon _{\beta _{k} }^{\frac{5}{2} }\left ( \left ( \left | x \right | ^{-1}\ast u_{2\beta _{k} } ^{2}  \right )  u_{1\beta _{k} } \right ) \left ( \varepsilon _{\beta _{k} }x+x _{\beta _{k} } \right ),
\end{split}
\end{equation}
where $\zeta =\varepsilon _{\beta _{k} }y+x _{\beta _{k} }$. In the same way, we obtain
   \begin{equation} \label{e93}
    \left ( \left ( \left | x \right | ^{-1}\ast w_{1\beta _{k} } ^{2}  \right )  w_{2\beta _{k} } \right )\left (x \right )=\varepsilon _{\beta _{k} }^{\frac{5}{2} }\left ( \left ( \left | x \right | ^{-1}\ast u_{1\beta _{k} } ^{2}  \right )  u_{2\beta _{k} } \right ) \left ( \varepsilon _{\beta _{k} }x+x _{\beta _{k} } \right ) ,
    \end{equation}

 \begin{equation} \label{e91}
  \left ( \left ( \left | x \right | ^{-1}\ast w_{i\beta _{k} } ^{2}  \right )  w_{i\beta _{k} } \right )\left (x \right )=\varepsilon _{\beta _{k} }^{\frac{5}{2} }\left ( \left ( \left | x \right | ^{-1}\ast u_{i\beta _{k} } ^{2}  \right )  u_{i\beta _{k} } \right ) \left ( \varepsilon _{\beta _{k} }x+x _{\beta _{k} } \right ), \ i=1,2.
    \end{equation}
 Since the nonnegative minimizer $\left ( u_{1\beta_{k} },u_{2\beta_{k} }  \right )$ satisfies system (\ref{eP1}) with the Lagrange multiplier $  \mu _{\beta_{k} } \in \mathbb{R} $, we then conclude from (\ref{e73}) and (\ref{e90})-(\ref{e91}) that $\left ( w_{1\beta_{k}  },w_{2\beta_{k}  }  \right )$ satisfies the following equations in $\mathbb{R} ^{3}  $,
    \begin{equation}\label{e95}
\begin{cases}
\begin{split}
\sqrt{- \triangle +  \varepsilon _{\beta _{k} }^{2} m^{2} }w_{1\beta _{k} } +\varepsilon _{\beta _{k} }&V_{1}(\varepsilon _{\beta _{k} }x+x _{\beta _{k} } )w_{1\beta _{k} }=\varepsilon _{\beta _{k} }\mu _{\beta _{k} } w_{1\beta _{k} }+a_{1}\phi _{w_{1\beta _{k} }}w_{1\beta _{k} }+\beta _{k}  \phi _{w_{2\beta _{k} }}w_{1\beta _{k} },\\
\sqrt{- \triangle +  \varepsilon _{\beta _{k} }^{2} m^{2} }w_{2\beta _{k} }+\varepsilon _{\beta _{k} }&V_{2}(\varepsilon _{\beta _{k} }x+x _{\beta _{k} } )w_{2\beta _{k} }=\varepsilon _{\beta _{k} }\mu _{\beta _{k} } w_{2\beta _{k} } +a_{2}\phi _{w_{2\beta _{k} }}w_{2\beta _{k} }+\beta _{k} \phi _{w_{1\beta _{k} }}w_{2\beta _{k} },
    \end{split}
\end{cases}
\end{equation}
where $\phi _{u}\left ( x \right ):=\left | x \right | ^{-1}\ast  u^{2}  $ for $x\in \mathbb{R} ^{3} $.
From (\ref{e95}), one can easily obtain
\begin{equation} \label{e103}
    \begin{split}
 \varepsilon _{\beta _{k} }\mu _{\beta _{k} }=&\displaystyle\sum_{i=1}^{2}\int\limits_{\mathbb{R} ^{3} }\left ( w_{i\beta _{k} }\sqrt{- \triangle +  \varepsilon _{\beta _{k} }^{2} m^{2} }w_{i\beta _{k} } +\varepsilon _{\beta _{k} }V_{i}(\varepsilon _{\beta _{k} }x+x _{\beta _{k} } )w_{i\beta _{k} }^{2} \right ) dx \\
 &-\sum_{i=1}^{2} a_{i}\int\limits_{\mathbb{R} ^{3} }\phi _{w_{i\beta _{k} }}w_{i\beta _{k} }^{2}dx-2\beta _{k} \int\limits_{\mathbb{R} ^{3} } \phi _{w_{2\beta _{k} }}w_{1\beta _{k} }^{2}dx.
    \end{split}
    \end{equation}
Furthermore, one can deduce from (\ref{e2}) and (\ref{e90})-(\ref{e93}) that
 \begin{equation} \label{e104}
    \begin{split}
    \varepsilon _{\beta _{k} }e\left ( a_{1},a_{2},\beta _{k}  \right )=&\varepsilon _{\beta _{k} }E_{a_{1}, a_{2}, \beta_{k}}\left ( u_{1\beta_{k} }, u_{2\beta_{k}} \right )\\
=&\displaystyle\sum_{i=1}^{2}\int\limits_{\mathbb{R} ^{3} }\left ( w_{i\beta _{k} }\sqrt{- \triangle +  \varepsilon _{\beta _{k} }^{2} m^{2} }w_{i\beta _{k} } +\varepsilon _{\beta _{k} }V_{i}(\varepsilon _{\beta _{k} }x+x _{\beta _{k} } )w_{i\beta _{k} }^{2} \right ) dx \\
 &-\sum_{i=1}^{2} \frac{a_{i}}{2} \int\limits_{\mathbb{R} ^{3} }\phi _{w_{i\beta _{k} }}w_{i\beta _{k} }^{2}dx-\beta _{k} \int\limits_{\mathbb{R} ^{3} } \phi _{w_{2\beta _{k} }}w_{1\beta _{k} }^{2}dx.
    \end{split}
    \end{equation}
Combining (\ref{e103}) and (\ref{e104}), we conclude from (\ref{e72}) and (\ref{e75}) that
    \begin{equation} \label{e94}
    \begin{split}
    \varepsilon _{\beta _{k} }\mu _{\beta _{k} }=& 2\varepsilon _{\beta _{k} }e\left ( a_{1},a_{2},\beta _{k}  \right )-\sum_{i=1}^{2}\int\limits_{\mathbb{R} ^{3}}\left ( w_{i\beta _{k} } \sqrt{- \triangle +  \varepsilon _{\beta _{k} }^{2} m^{2} }w_{i\beta _{k} } +\varepsilon _{\beta _{k} }V_{i}(\varepsilon _{\beta _{k} }x+x _{\beta _{k} } )w_{i\beta _{k} } ^{2} \right ) dx\\
    &\to -1, \quad {\rm as}\ k\to \infty .
    \end{split}
    \end{equation}
Taking the weak limit of (\ref{e95}), one can find that $\left (w _{10}, w _{20} \right ) $ satisfies
        \begin{equation} \label{e96}
    \begin{cases}
    \sqrt{- \triangle}w_{10 } +w_{10} =a_{1}\phi _{w_{10 }}w_{10 }+\beta ^{\ast }  \phi _{w_{20 }}w_{10},\\
\sqrt{- \triangle }w_{20 }+w_{20 }=a_{2}\phi _{w_{20}}w_{20 }+\beta^{\ast } \phi _{w_{10 }}w_{20}.
\end{cases}
\end{equation}
Now, we come to show that there exist $r_{2},\rho _{2}>0 $ such that
        \begin{equation} \label{e97}
\liminf _{k\to \infty }\int\limits_{B_{r_{2} }(0) }\left | w_{2\beta _{k} }\left ( x \right )  \right |  ^{2}dx\ge \rho_{2} > 0.
\end{equation}
Indeed, it suffices to prove $w_{20} \ne 0$. Arguing by contradiction that $w_{20} = 0$, we then see from (\ref{e96}) that $w_{10} \ge 0$ satisfies
        \begin{equation} \label{e98}
 \sqrt{- \triangle}w_{10 } +w_{10} =a_{1}\phi _{w_{10 }}w_{10 } \quad {\rm in}\ \mathbb{R} ^{3}.
 \end{equation}
 By a simple rescaling, $\sqrt{a_{1} } w_{10 }$ satisfies (\ref{e8}). Note that $Q$ is a ground state of (\ref{e8}), Lemma \ref{le1} obviously implies
 $$\left \| w_{10} \right \|  _{2}^{2}\ge \frac{a^{\ast } }{a_{1} }> 1,$$
which leads to a contradiction. In fact, it follows from (\ref{e75}) and Fatou's Lemma that
$$\left \| w_{10} \right \|  _{2}^{2}\le \liminf _{k\to \infty }\left ( \left \| w_{1\beta _{k} } \right \|  _{2}^{2}+\left \| w_{2\beta _{k} } \right \|  _{2}^{2}  \right )=1.$$
Therefore, (\ref{e97}) holds. Combining (\ref{e87}) and (\ref{e97}), we also obtain (\ref{e74}), which finishes (i).
\end{enumerate}
\begin{enumerate}[(\romannumeral2)]
\item  To prove (ii), from (\ref{e89}), it suffices to show that $x_{0}\in \mathcal{V}  $. Arguing by contradiction that $x_{0}\notin  \mathcal{V}$, without loss of generality, we suppose $V_{1}(x_{0})\ne 0 $, which indicates $V_{1}(x_{0}) >0$. We then deduce from Fatou's Lemma and (\ref{e74}) that
    $$\liminf _{k\to \infty } \sum_{i=1}^{2} \int\limits_{\mathbb{R} ^{3} }V_{i}(\varepsilon _{\beta _{k} } x+x_{\beta _{k} })\left | w_{i\beta _{k} }   \right | ^{2}dx\ge V_{1}(x_{0} ) \liminf _{k\to \infty }\int\limits_{B_{r_{0} }(0) }\left | w_{1\beta _{k} }  \right |  ^{2}dx\ge V_{1}(x_{0} )\rho > 0,$$
    which contradicts (\ref{e88}) so that (ii) follows.
\end{enumerate}
\begin{enumerate}[(\romannumeral3)]
\item Using Theorem 6.1 in \cite{Zhou}, we know that the solution $\left ( w_{10},w_{20}   \right ) $ of system (\ref{e96}) satisfies the following type Pohozaev identity
        \begin{equation} \label{e99}
        \begin{split}
&\int\limits_{\mathbb{R} ^{3} }\left ( \left | (-\triangle )^{\frac{1}{4} }w_{10 }   \right | ^{2}+\left | (-\triangle )^{\frac{1}{4} }w_{20}   \right | ^{2} \right ) dx+\frac{3}{2}\int\limits_{\mathbb{R} ^{3} }\left (  w_{10}^{2}+ w_{20}^{2}  \right )dx\\
=&\frac{5}{4}\int\limits_{\mathbb{R} ^{3} }\left ( a_{1}\phi _{w_{10}}w_{10 }^{2}+a_{2}\phi _{w_{20}}w_{20 }^{2}+2\beta ^{\ast }  \phi _{w_{10 }}w_{20 }^{2}  \right ) dx.
\end{split}
 \end{equation}
 Note from (\ref{e75}) that
 \begin{equation} \label{e100}
 \int\limits_{\mathbb{R} ^{3} }\left ( \left | (-\triangle )^{\frac{1}{4} }w_{10 }   \right | ^{2}+\left | (-\triangle )^{\frac{1}{4} }w_{20 }   \right | ^{2} \right ) dx\le 1\quad {\rm and}\ \int\limits_{\mathbb{R} ^{3} } (w_{10 } ^{2}+w_{20}  ^{2})dx\le 1.
  \end{equation}
  {\re Moreover, we can infer from (\ref{e96}) that
   \begin{equation} \label{e1101}
        \begin{split}
&\int\limits_{\mathbb{R} ^{3} }\left ( \left | (-\triangle )^{\frac{1}{4} }w_{10 }   \right | ^{2}+\left | (-\triangle )^{\frac{1}{4} }w_{20}   \right | ^{2} \right ) dx+\int\limits_{\mathbb{R} ^{3} }\left (  w_{10}^{2}+ w_{20}^{2}  \right )dx\\
=&\int\limits_{\mathbb{R} ^{3} }\left ( a_{1}\phi _{w_{10}}w_{10 }^{2}+a_{2}\phi _{w_{20}}w_{20 }^{2}+2\beta ^{\ast }  \phi _{w_{10 }}w_{20 }^{2}  \right ) dx.
\end{split}
 \end{equation}
 Applying the Pohozaev identity (\ref{e99}), (\ref{e1101}) leads to
    \begin{equation} \label{e1102}
2\int\limits_{\mathbb{R} ^{3} }\left (  w_{10}^{2}+ w_{20}^{2}  \right )dx
=\int\limits_{\mathbb{R} ^{3} }\left ( a_{1}\phi _{w_{10}}w_{10 }^{2}+a_{2}\phi _{w_{20}}w_{20 }^{2}+2\beta ^{\ast }  \phi _{w_{10 }}w_{20 }^{2}  \right ) dx.
 \end{equation}
  Together with the  refined Gagliardo-Nirenberg inequality (\ref{e27}) and Lemma \ref{b3}, we deduce from (\ref{e100}), (\ref{e1101}) and (\ref{e1102}) that}
   \begin{equation} \label{e101}
        \begin{split}
        &\int\limits_{\mathbb{R} ^{3} }\left ( \left | (-\triangle )^{\frac{1}{4} }w_{10 }   \right | ^{2}+\left | (-\triangle )^{\frac{1}{4} }w_{20 }   \right | ^{2} \right ) dx=\int\limits_{\mathbb{R} ^{3} } (w_{10 } ^{2}+w_{20}  ^{2})dx\\
        =&\frac{1}{2}\int\limits_{\mathbb{R} ^{3} }\left ( a_{1}\phi _{w_{10}}w_{10 }^{2}+a_{2}\phi _{w_{20}}w_{20 }^{2}+2\beta ^{\ast }  \phi _{w_{10 }}w_{20 }^{2}  \right ) dx\\
        =&\frac{a^{\ast }}{2}D\left ( w_{10}^{2}+w_{20}^{2},w_{10}^{2}+w_{20}^{2} \right )-\frac{1}{2}\left [ \sqrt{\left ( a^{\ast }-a_{1} \right )D\left ( w_{10}^2,w_{10}^2 \right )    }-\sqrt{\left ( a^{\ast }-a_{2} \right )D\left ( w_{20}^2,w_{20}^2 \right )   }    \right ]^2\\
&-\sqrt{\left ( a^{\ast }-a_{1} \right )\left ( a^{\ast }-a_{2} \right )   } \left [  \sqrt{D\left ( w_{10}^2,w_{10}^2 \right )D\left ( w_{20}^2,w_{20}^2 \right )}-D\left ( w_{10}^2,w_{20}^2 \right )  \right ]\\
\le &\frac{a^{\ast }}{2}D\left ( w_{10}^{2}+w_{20}^{2},w_{10}^{2}+w_{20}^{2} \right )\\
\le & \int\limits_{\mathbb{R} ^{3} }\left ( \left | (-\triangle )^{\frac{1}{4} }w_{10 }   \right | ^{2}+\left | (-\triangle )^{\frac{1}{4} }w_{20 }   \right | ^{2} \right ) dx\int\limits_{\mathbb{R} ^{3} } (w_{10 } ^{2}+w_{20}  ^{2})dx\\
\le& \int\limits_{\mathbb{R} ^{3} }\left ( \left | (-\triangle )^{\frac{1}{4} }w_{10 }   \right | ^{2}+\left | (-\triangle )^{\frac{1}{4} }w_{20 }   \right | ^{2} \right ) dx,
\end{split}
 \end{equation}
 which suggests that the aforementioned equalities hold. Therefore, based on Lemma \ref{b3} and Lemma \ref{le2}, we can deduce that
     \begin{equation} \label{e812}
\left ( w_{10}(x), w_{20}(x) \right )=\left ( \tau sin\theta  Q\left (\alpha x-\eta_{0}  \right ) ,\tau cos\theta Q(\alpha x-\eta_{0} ) \right ), \quad w_{10}^{2}(x)=\kappa w_{20}^{2}(x),
  \end{equation}
with some $\tau>0, \alpha>0, \theta \in \left [ 0,2\pi  \right ), \eta_{0} \in \mathbb{R} ^{3}$ and $\kappa>0$. Actually, it follows from (\ref{e101}) that
   \begin{equation} \label{e810}
   \int\limits_{\mathbb{R} ^{3} }\left ( \left | (-\triangle )^{\frac{1}{4} }w_{10 }   \right | ^{2}+\left | (-\triangle )^{\frac{1}{4} }w_{20 }   \right | ^{2} \right ) dx= \int\limits_{\mathbb{R} ^{3} } (w_{10 } ^{2}+w_{20}  ^{2})dx= 1,
     \end{equation}
   which implies $\alpha =1$ and $\tau =\frac{1}{\left \| Q \right \|_{2} }$ in (\ref{e812}). Besides, (\ref{e101}) also gives
    \begin{equation} \label{e811}
   \sqrt{\left ( a^{\ast }-a_{1} \right )D\left ( w_{10}^2,w_{10}^2 \right )    }=\sqrt{\left ( a^{\ast }-a_{2} \right )D\left ( w_{20}^2,w_{20}^2 \right )   },
          \end{equation}
   which infers $\kappa =\sqrt{\frac{a^{\ast }-a_{2}}{a^{\ast }-a_{1}} } $ in (\ref{e812}) and through simple calculations we thus conclude that
 \begin{equation} \label{e102}
w_{10}(x)=\frac{\sqrt{\gamma } }{\left \| Q \right \| _{2} }Q(x-\eta_{0} ) \quad {\rm and} \quad w_{20}(x)=\frac{\sqrt{1-\gamma } }{\left \| Q \right \| _{2} }Q(x-\eta_{0} ),
            \end{equation}
where $\gamma $ is given by (\ref{e48}).
 Moreover, one can further derive from (\ref{e75}) and { (\ref{e810})} that
 $$ \lim _{k\to \infty }\int\limits_{\mathbb{R} ^{3} }\left ( \left | (-\triangle )^{\frac{1}{4} }w_{i\beta _{k}  }   \right | ^{2}+w_{i\beta _{k} } ^{2}\right ) dx=\int\limits_{\mathbb{R} ^{3} }\left (  \left | (-\triangle )^{\frac{1}{4} }w_{i0 }   \right | ^{2} +w_{i0}  ^{2}\right )dx,\quad i=1,2.$$
 Combining with the fact that $w_{i\beta _{k} }\rightharpoonup w_{i0} $ in $H^{\frac{1}{2} }\left (\mathbb{ R} ^{3}  \right ) $, we thus obtain that $$w_{i\beta _{k} }\to  w_{i0} \quad {\rm strongly ~~in}\ ~~ H^{\frac{1}{2} }\left (\mathbb{ R} ^{3}  \right ), \quad i=1,2. $$
From (\ref{e102}), we show that (\ref{e80}) holds, and the proof of the lemma is done.
\end{enumerate}
 \end{proof}
\noindent{\bf Proof of Theorem \ref{T2}:} Notice that from (\ref{e95}), one can obtain
\begin{equation} \label{e205}
\sqrt{-\triangle }w_{i\beta _{k}}(x)-c_{i\beta _{k}}(x)w_{i\beta _{k}}(x)\le 0 \quad {\rm in} \ \mathbb{R} ^{3},\quad i=1,2,
\end{equation}
where $$c_{1\beta _{k}}:=a_{1} \phi _{w_{1\beta _{k}}}+\beta _{k}\phi _{w_{2\beta _{k}}} \quad {\rm and} \quad c_{2\beta _{k}}:=a_{2} \phi _{w_{2\beta _{k}}}+\beta _{k}\phi _{w_{1\beta _{k}}}, \quad i=1,2.     $$
Applying the non-local De Giorgi-Nash-Moser theory (see Theorem 1.1 in \cite{Imbesi} or Theorem 5.4 in \cite{Kim}) to globally nonnegative weak subsolution $w_{i\beta _{k}}(i=1,2)$ of (\ref{e205}), we then have
\begin{equation} \label{e206}
\sup_{B_{1}(\xi _{0}) }w_{i\beta _{k}}\le C\left ( \int\limits_{B_{2}(\xi _{0})}\left | w_{i\beta _{k}} \right |^{2}dx  \right ) ^{\frac{1}{2} }, \quad i=1,2,
\end{equation}
where $C$ is a positive constant and $\xi _{0}$ is an arbitrary point in $\mathbb{R} ^{3} $. Moreover, from  (\ref{e80}) we can deduce that
\begin{equation} \label{e207}
\int\limits_{\left | x \right |> r }\left | w_{i\beta _{k}} \right |^{2} dx\to 0 \quad {\rm as} \ r\to \infty \ {\rm uniformly \ in }\ k,\ i=1,2.
\end{equation}
Combining (\ref{e206}) and (\ref{e207}), we thus find that
\begin{equation} \label{e208}
w_{i\beta _{k}}(x)\to 0 \quad {\rm as} \ \left | x \right | \to \infty \ {\rm uniformly \ in }\ k,\ i=1,2.
\end{equation}
Therefore, we conclude that $w_{i\beta _{k}}(i=1,2)$ has at least one global maximum point.\\

Let $z_{i\beta _{k}}$ be a global maximum point of $u_{i\beta _{k}}$ for $i=1,2$, from (\ref{e73}), $w_{i\beta _{k}}$ thus attains its global maximum at the point $\frac{z_{i\beta _{k}}-x_{\beta _{k}} }{\varepsilon _{\beta _{k}} } $. Hence, one can derive from (\ref{e74}) and (\ref{e208}) that
$$\limsup_{k \to \infty} \frac{\left | z_{i\beta _{k}}-x_{\beta _{k}}  \right | }{\varepsilon _{\beta _{k}} }< \infty , \quad i=1,2.$$
Going if necessary to a subsequence, there exists $\eta_{i}\in \mathbb{R} ^{3}$ such that
\begin{equation} \label{e209}
\lim_{k \to \infty} \frac{z_{i\beta _{k}}-x_{\beta _{k}} }{\varepsilon _{\beta _{k}} }=\eta_{i}, \quad i=1,2.
\end{equation}
Define
\begin{equation} \label{e210}
\bar{w} _{i\beta _{k}}(x):=\varepsilon _{\beta _{k} }^{\frac{3}{2} }u_{i\beta _{k} }\left ( \varepsilon _{\beta _{k} }x+z_{i\beta _{k} } \right )=w_{i\beta _{k} }(x+\frac{z_{i\beta _{k}}-x_{\beta _{k}} }{\varepsilon _{\beta _{k}} }) , \quad i=1,2.
\end{equation}
From Lemma \ref{b3.2} (iii), we obtain
\begin{equation} \label{e211}
\begin{cases}
\lim_{k  \to \infty  }\bar{w} _{1\beta _{k} }\left ( x \right )=\bar{w} _{10}(x):= \frac{\sqrt{\gamma  }}{\left \| Q \right \| _{2} }Q\left (  x+\eta_{1} -\eta_{0}   \right ),\\
\lim_{k  \to \infty  }\bar{w} _{2\beta _{k} }\left ( x \right )=\bar{w} _{20}(x):= \frac{\sqrt{1-\gamma  }}{\left \| Q \right \| _{2} }Q\left (  x+\eta_{2}-\eta_{0}   \right ),
\end{cases}
\end{equation}
strongly in $H^{\frac{1}{2} }\left ( \mathbb{R} ^{3}  \right ) $. Note from (\ref{e210}) that the origin is a critical point of $\bar{w} _{i\beta _{k}}$, then it is also a critical point of $\bar{w} _{i0}$ for $i=1,2$. We thus deduce that $\bar{w} _{i0}(i=1,2)$ are spherically symmetric about the origin due to the fact that the ground state $Q(x)$ is radially symmetric and $Q(r)$ decreasing  with $r=|x|$, which implies $\eta_{1}=\eta_{2}=\eta_{0} $ in (\ref{e211}). Therefore, (\ref{e201}) and (\ref{e203}) immediately follow from (\ref{e209}) and (\ref{e211}). Furthermore, (\ref{e209}) and Lemma \ref{b3.2} (ii) imply that (\ref{e200}) holds. Finally,
(\ref{e204}) can be derived by applying (\ref{e76}) and (\ref{e78}) so that the proof of Theorem \ref{T2} has been done.
\section{Explicit concentration behavior}
We are concerned with the concentration behavior of the minimizers for (\ref{e1}) under explicit trapping potentials, i.e., the polynomial potentials in this section, where some delicate estimates of the associated functionals will be also given.
\begin{lemma} Suppose that $V_{i}(x)(i=1,2)$ satisfy (\ref{e212}) and (\ref{e213}). If it holds $0<q_{0}<1$ when $m\ne 0$ and $0<q_{0}<\frac{5}{2} $ when $m=0$, then for any given $0<a_{1},a_{2}<a^{\ast } $ and $0<\beta < \beta ^{\ast } $, we have
\begin{equation} \label{e224}
\liminf_{\beta  \to \beta ^{\ast }} \frac{e(a_{1},a_{2},\beta )}{\left ( \beta ^{\ast }-\beta \right )^{\frac{q_{0}}{q_{0}+1} }} \le
\frac{q_{0}+1}{q_{0}a^{\ast }}\left ( q_{0}\lambda _{0}  \right )^{\frac{1}{q_{0}+1} }\left [ 2\gamma (1-\gamma ) \right ]^{\frac{q_{0}}{q_{0}+1} },
\end{equation}
where $\gamma \in \left ( 0,1 \right )$, $q_{0}>0$ and $\lambda _{0}>0$ are given in (\ref{e48}), (\ref{e216}) and (\ref{e218}), respectively.
\end{lemma}
\begin{proof}
Let $\left ( \phi _{1\sigma } ,\phi _{2\sigma } \right ) $ be the function given by (\ref{e36}) with
\begin{equation} \label{e824}
\sigma =\left [ \frac{q_{0}\lambda _{0}}{2\gamma (1-\gamma )(\beta ^{\ast }-\beta )}  \right ]^{\frac{1}{q_{0}+1} }.
 \end{equation}
 At the same time, set $\theta =\gamma$ and take $x_{0}=x_{j_{0}}$, where $x_{j_{0}}\in Z_{0}$ with $Z_{0}$ given by (\ref{e219}). Applying (\ref{e47}), one can obtain
\begin{equation} \label{e225}
\begin{split}
&\sum_{i=1}^{2}\int\limits_{\mathbb{R} ^{3} }\phi _{i\sigma }\sqrt{- \triangle + m^{2} }\phi _{i\sigma }dx-\frac{1}{2} \left [a_{1}D\left ( \phi _{1\sigma } ^{2},\phi _{1\sigma } ^{2}  \right )+ a_{2}D\left ( \phi _{2\sigma } ^{2},\phi _{2\sigma } ^{2}  \right )+2\beta D\left ( \phi _{1\sigma } ^{2},\phi _{2\sigma } ^{2}  \right )\right ] \\
\le &\frac{2\sigma}{a^{\ast }} \gamma \left ( 1-\gamma  \right ) \left ( \beta ^{\ast } - \beta  \right )+O\left ( \sigma ^{-\frac{5}{2} }  \right )+2m^{2}O\left ( \sigma ^{-1}  \right ).
\end{split}
\end{equation}
Note that, when $m\ne 0$, we then know from $0<q_{0}<1$ that
\begin{equation} \label{e826}
\frac{2\sigma}{a^{\ast }} \gamma \left ( 1-\gamma  \right ) \left ( \beta ^{\ast } - \beta  \right )+O\left ( \sigma ^{-\frac{5}{2} }  \right )+2m^{2}O\left ( \sigma ^{-1}  \right )=\frac{2\sigma}{a^{\ast }} \gamma \left ( 1-\gamma  \right ) \left ( \beta ^{\ast } - \beta  \right )+o(\sigma ^{-q_{0}}),
\end{equation}
as $\sigma \to  \infty $. Moreover, when $m=0$, it then follows from $0<q_{0}<\frac{5}{2} $ that (\ref{e826}) still holds. Therefore, combining (\ref{e225}) and (\ref{e826}), one can deduce from (\ref{e824}) that
\begin{equation} \label{e827}
\begin{split}
&\sum_{i=1}^{2}\int\limits_{\mathbb{R} ^{3} }\phi _{i\sigma }\sqrt{- \triangle + m^{2} }\phi _{i\sigma }dx-\frac{1}{2} \left [a_{1}D\left ( \phi _{1\sigma } ^{2},\phi _{1\sigma } ^{2}  \right )+ a_{2}D\left ( \phi _{2\sigma } ^{2},\phi _{2\sigma } ^{2}  \right )+2\beta D\left ( \phi _{1\sigma } ^{2},\phi _{2\sigma } ^{2}  \right )\right ]  \\
\le &\frac{1}{a^{\ast }}\left ( q_{0}\lambda _{0} \right )^{\frac{1}{q_{0}+1} }  \left [ 2\gamma (1-\gamma )  \right ]^{\frac{q_{0}}{q_{0}+1} }\left ( \beta ^{\ast }-\beta  \right )^{\frac{q_{0}}{q_{0}+1} }+o(\left ( \beta ^{\ast }-\beta  \right )^{\frac{q_{0}}{q_{0}+1} }) \quad {\rm as}\ \beta \to  \beta ^{\ast }.
\end{split}
\end{equation}
On the other hand, without loss of generality, we suppose $q_{0}=q_{1j_{0}}< q_{2j_{0}}$, hence we can deduce from (\ref{e36}) and (\ref{e215})-(\ref{e219}) that
\begin{equation} \label{e226}
\begin{split}
&\int\limits_{\mathbb{R} ^{3} }V_{1}(x)\phi _{1\sigma }^{2}dx+\int\limits_{\mathbb{R} ^{3} }V_{2}(x)\phi _{2\sigma }^{2}dx \\
=&\sigma ^{-q_{0}}\int\limits_{\mathbb{R} ^{3} }\frac{A_{\sigma }^{2}\gamma   V_{1}(\sigma ^{-1} x+x_{j_{0}} ) }{a^{\ast }}\sigma ^{q_{0}}\varphi ^{2}(\sigma ^{-1}x)Q^{2}(x)dx\\
&+\sigma ^{-q_{0}}\int\limits_{\mathbb{R} ^{3} }\frac{A_{\sigma }^{2}(1-\gamma)   V_{2}(\sigma ^{-1} x+x_{j_{0}} ) }{a^{\ast }}\sigma ^{q_{0}}\varphi ^{2}(\sigma ^{-1}x)Q^{2}(x)dx \\
=&\sigma ^{-q_{0}}\int\limits_{\mathbb{R} ^{3} }\frac{A_{\sigma }^{2}\gamma   V_{1}(\sigma ^{-1} x+x_{j_{0}} ) }{a^{\ast }\left | \sigma ^{-1}x \right |^{q_{1j_{0}}} }\left | x \right |^{q_{1j_{0}}}\varphi ^{2}(\sigma ^{-1}x)Q^{2}(x)dx\\
&+\sigma ^{-q_{0}}\int\limits_{\mathbb{R} ^{3} }\frac{A_{\sigma }^{2}(1-\gamma)   V_{2}(\sigma ^{-1} x+x_{j_{0}} ) }{a^{\ast }\left | \sigma ^{-1}x \right |^{q_{2j_{0}}}}\sigma ^{q_{0}-q_{2j_{0}}}\left | x \right |^{q_{2j_{0}}}\varphi ^{2}(\sigma ^{-1}x)Q^{2}(x)dx\\
 =& \frac{\lambda _{0} }{a^{\ast }\sigma ^{q_{0}}}+o(\sigma ^{-q_{0}}) \quad {\rm as}\ \sigma \to \infty .
\end{split}
\end{equation}
Therefore, from (\ref{e824}), we thus infer
\begin{equation} \label{e228}
\begin{split}
&\int\limits_{\mathbb{R} ^{3} }V_{1}(x)\phi _{1\sigma }^{2}dx+\int\limits_{\mathbb{R} ^{3} }V_{2}(x)\phi _{2\sigma }^{2}dx\\
=&\frac{1}{q_{0}a^{\ast }} \left ( q_{0}\lambda _{0} \right )^{\frac{1}{q_{0}+1} }  \left [ 2\gamma (1-\gamma )  \right ]^{\frac{q_{0}}{q_{0}+1} }\left ( \beta ^{\ast }-\beta  \right )^{\frac{q_{0}}{q_{0}+1} }+o(\left ( \beta ^{\ast }-\beta  \right )^{\frac{q_{0}}{q_{0}+1} }) \quad {\rm as}\ \beta \to  \beta ^{\ast }.
\end{split}
\end{equation}
Combining (\ref{e827}) and (\ref{e228}),  we then derive the estimate (\ref{e224}), as claimed.
\end{proof}

\noindent{\bf Proof of Theorem \ref{T3}:} Let $ \left \{ \left ( u_{1\beta _{k}},u_{2\beta _{k}} \right )  \right \} $ be the convergent subsequence obtained in Theorem \ref{T2}. Applying the refined Gagliardo-Nirenberg inequality (\ref{e27}) and Lemma \ref{b3} to (\ref{e51}), we then obtain for any $k\in \mathbb{N} $
\begin{equation} \label{e229}
\begin{split}
e(a_{1},a_{2},\beta _{k})&\ge \left ( \beta ^{\ast }-\beta _{k} \right )D\left ( u_{1\beta _{k}}^{2},u_{2\beta _{k}}^{2} \right )+\sum_{i=1}^{2}\int\limits_{\mathbb{R}^{3} }V_{i}(x)\left | u_{i\beta _{k}} \right |^{2}dx\\
&=\frac{\beta ^{\ast }-\beta _{k}}{\varepsilon _{\beta _{k}}}D\left (\bar{ w} _{1\beta _{k}}^{2},\bar{w} _{2\beta _{k}}^{2} \right )+ \sum_{i=1}^{2}\int\limits_{\mathbb{R}^{3} }V_{i}(\varepsilon _{\beta _{k}}x+z_{i\beta _{k}})\left | \bar{w} _{i\beta _{k}} \right |^{2}dx,
\end{split}
\end{equation}
where the second equality comes from (\ref{e210}). Using (\ref{e78}) and (\ref{e210}), we find that
\begin{equation} \label{e230}
D\left (\bar{ w} _{1\beta _{k}}^{2},\bar{w} _{2\beta _{k}}^{2} \right )=\frac{2\gamma (1-\gamma )}{a^{\ast }}+o(1) \quad {\rm as} \ k\to \infty .
\end{equation}
Note that $x_{0}\in \mathcal{V}$ in (\ref{e200}), hence there exists $1\le j_{0}\le b$ such that $x_{0}=x_{j_{0}}$. Without loss of generality, we may assume that $q_{j_{0}}=q_{1j_{0}}< q_{2j_{0}}$. Next, we claim that
\begin{equation} \label{e231}
\frac{z_{1\beta _{k}}-x_{j_{0}}}{\varepsilon _{\beta _{k}}} \ {\rm is\ bounded\ as}\ k\to \infty,
\end{equation}
and
\begin{equation} \label{e232}
q_{j_{0}}=q_{0}.
\end{equation}
Suppose by contradiction that either $q_{j_{0}}<q_{0}$ or (\ref{e231}) is false. Then, for any given $M>0$, one can infer from Fatou's Lemma that
\begin{equation} \label{e233}
\begin{split}
&\liminf_{k \to \infty} \varepsilon _{\beta _{k}}^{-q_{0}}\int\limits_{\mathbb{R} ^{3}}V_{1}(\varepsilon _{\beta _{k}}x+z_{1\beta _{k}})\left | \bar{w} _{1\beta _{k}} \right |^{2}dx\\
\ge& \int\limits_{\mathbb{R} ^{3}}\liminf_{k \to \infty}\varepsilon _{\beta _{k}}^{q_{j_{0}}-q_{0}}\left ( \frac{V_{1}(\varepsilon _{\beta _{k}}x+z_{1\beta _{k}})}{\left | \varepsilon _{\beta _{k}}x+z_{1\beta _{k}}-x_{j_{0}} \right |^{q_{j_{0}}} }\left | x+\frac{z_{1\beta _{k}}-x_{j_{0}}}{\varepsilon _{\beta _{k}}} \right |^{q_{j_{0}}}\left | \bar{w} _{1\beta _{k}} \right |^{2}     \right )dx \ge M.
\end{split}
\end{equation}
From (\ref{e229}), (\ref{e230}) and (\ref{e233}), we thus deduce that
\begin{equation} \label{e234}
e(a_{1},a_{1},\beta _{k})\ge \frac{2\gamma (1-\gamma )}{a^{\ast }}(\beta ^{\ast }-\beta_{k} )\varepsilon _{\beta _{k}}^{-1}+M\varepsilon _{\beta _{k}}^{q_{0}}\ge CM^{\frac{1}{q_{0}+1} }(\beta ^{\ast }-\beta_{k} )^{\frac{q_{0}}{q_{0}+1} } \quad {\rm for\ large}\ k.
\end{equation}
In fact, the second inequality above holds due to the Young inequality and $C$ is a suitable positive constant. Note that $M$ is arbitrary, the claim is thus verified by contradiction between (\ref{e224}) and (\ref{e234}).
Following (\ref{e203}) and (\ref{e231}), after passing to a subsequence, there exists $z_{0}\in \mathbb{R} ^{3}$ such that
\begin{equation} \label{e235}
\lim_{k \to \infty} \frac{z_{i\beta _{k}}-x_{j_{0}}}{\varepsilon _{\beta _{k}}}=z_{0} \quad {\rm for }\ i=1,2.
\end{equation}
Combining this with (\ref{e211}), we have
\begin{equation} \label{e236}
\begin{split}
&\lim_{k \to \infty}\varepsilon _{\beta _{k}}^{-q_{0}}\left ( \int\limits_{\mathbb{R} ^{3}}V_{1}(\varepsilon _{\beta _{k}}x+z_{1\beta _{k}})\left | \bar{w} _{1\beta _{k}} \right |^{2}dx+\int\limits_{\mathbb{R} ^{3}}V_{2}(\varepsilon _{\beta _{k}}x+z_{1\beta _{k}})\left | \bar{w} _{1\beta _{k}} \right |^{2}dx \right )\\
=& \lim_{k \to \infty} \int\limits_{\mathbb{R} ^{3}}\frac{V_{1}(\varepsilon _{\beta _{k}}x+z_{1\beta _{k}})}{\left | \varepsilon _{\beta _{k}}x+z_{1\beta _{k}}-x_{j_{0}} \right |^{q_{1j_{0}}} }\left | x+\frac{z_{1\beta _{k}}-x_{j_{0}}}{\varepsilon _{\beta _{k}}} \right |^{q_{1j_{0}}}\left | \bar{w} _{1\beta _{k}} \right |^{2}dx+\\
&\lim_{k \to \infty} \varepsilon _{\beta _{k}}^{q_{2j_{0}}-q_{0}}\int\limits_{\mathbb{R} ^{3}}\frac{V_{2}(\varepsilon _{\beta _{k}}x+z_{2\beta _{k}})}{\left | \varepsilon _{\beta _{k}}x+z_{2\beta _{k}}-x_{j_{0}} \right |^{q_{2j_{0}}} }\left | x+\frac{z_{2\beta _{k}}-x_{j_{0}}}{\varepsilon _{\beta _{k}}} \right |^{q_{2j_{0}}}\left | \bar{w} _{2\beta _{k}} \right |^{2}dx\\
=&\lim_{x \to x_{j_{0}}} \frac{\gamma V_{1}(x)}{a^{\ast }\left | x-x_{j_{0}} \right |^{q_{1j_{0}}} }\int\limits_{\mathbb{R} ^{3}} \left | y+z_{0} \right |^{q_{1j_{0}}}Q^{2}(y)dy.
\end{split}
\end{equation}
Note that
\begin{equation} \label{e237}
\int\limits_{\mathbb{R} ^{3}}\left | y+z_{0} \right |^{q_{1j_{0}}}Q^{2}(y)dy\ge \int\limits_{\mathbb{R} ^{3}}\left | y \right |^{q_{1j_{0}}}Q^{2}(y)dy,
\end{equation}
where the equality holds if and only if $z_{0}=0$. Indeed, since $Q$ is radially symmetric and decreasing, we have
$$\int\limits_{\mathbb{R} ^{3}}\left | y+z_{0} \right |^{q_{1j_{0}}}Q^{2}(y)dy=\int\limits_{\mathbb{R} ^{3}}\left | y-z_{0} \right |^{q_{1j_{0}}}Q^{2}(y)dy,$$
then one can obtain
\begin{equation} \label{e238}
\begin{split}
&\int\limits_{\mathbb{R} ^{3}}\left | y+z_{0} \right |^{q_{1j_{0}}}Q^{2}(y)dy- \int\limits_{\mathbb{R} ^{3}}\left | y \right |^{q_{1j_{0}}}Q^{2}(y)dy\\
=&\frac{1}{2}\int\limits_{\mathbb{R} ^{3}} \left ( \left | y+z_{0} \right |^{q_{1j_{0}}}-\left | y \right |^{q_{1j_{0}}} \right )Q^{2}(y)dy+\frac{1}{2}\int\limits_{\mathbb{R} ^{3}} \left ( \left | y-z_{0} \right |^{q_{1j_{0}}}-\left | y \right |^{q_{1j_{0}}} \right )Q^{2}(y)dy\\
= &\frac{1}{2}\int\limits_{\mathbb{R} ^{3}} \left ( \left | y+z_{0} \right |^{q_{1j_{0}}}-\left | y \right |^{q_{1j_{0}}} \right )Q^{2}(y)dy+\frac{1}{2}\int\limits_{\mathbb{R} ^{3}} \left ( \left | y \right |^{q_{1j_{0}}}-\left | y+z_{0} \right |^{q_{1j_{0}}} \right )Q^{2}(y+z_{0})dy\\
=&\frac{1}{2}\int\limits_{\mathbb{R} ^{3}} \left ( \left | y+z_{0} \right |^{q_{1j_{0}}}-\left | y \right |^{q_{1j_{0}}} \right )\left ( Q^{2}(y)-Q^{2}(y+z_{0}) \right ) dy\ge 0,
\end{split}
\end{equation}
which implies that (\ref{e237}) holds. Therefore, from definition of $\lambda _{j_{0}}$ and $\lambda _{0}$, one can find that
\begin{equation} \label{e239}
\begin{split}
\lim_{x \to x_{j_{0}}} \frac{\gamma V_{1}(x)}{a^{\ast }\left | x-x_{j_{0}} \right |^{q_{1j_{0}}} }\int\limits_{\mathbb{R} ^{3}} \left | y+z_{0} \right |^{q_{1j_{0}}}Q^{2}(y)dy&\ge \lim_{x \to x_{j_{0}}} \frac{\gamma V_{1}(x)}{a^{\ast }\left | x-x_{j_{0}} \right |^{q_{1j_{0}}} }\int\limits_{\mathbb{R} ^{3}} \left | y\right |^{q_{1j_{0}}}Q^{2}(y)dy\\
&=\frac{\lambda _{j_{0}}}{a^{\ast }}\ge \frac{\lambda _{0}}{a^{\ast }},
\end{split}
\end{equation}
where the equalities hold if and only if $z_{0}=0$ and $x_{j_{0}} \in  Z_{0}$. Applying (\ref{e229}),(\ref{e230}),(\ref{e236}) and (\ref{e239}), we further obtain
\begin{equation} \label{e240}
e(a_{1},a_{2},\beta _{k})\ge \left ( \frac{2\gamma (1-\gamma )}{a^{\ast }}+o(1) \right ) \left ( \beta ^{\ast }-\beta _{k} \right ) \varepsilon _{\beta _{k}}^{-1}+ \left ( \frac{\lambda _{0}}{a^{\ast }}+o(1) \right )\varepsilon _{\beta _{k}}^{q_{0}}  \quad {\rm as} \ k\to \infty .
\end{equation}
Therefore, one can verify by taking the infimum of the above inequality that
\begin{equation} \label{e241}
\liminf_{k \to \infty} \frac{e(a_{1},a_{2},\beta _{k})}{\left ( \beta ^{\ast }-\beta _{k} \right )^{\frac{q_{0}}{q_{0}+1} }} \ge \frac{q_{0}+1}{q_{0}a^{\ast }}\left ( q_{0}\lambda _{0}  \right )^{\frac{1}{q_{0}+1} } \left [ 2\gamma (1-\gamma ) \right ]^{\frac{q_{0}}{q_{0}+1} },
\end{equation}
where the equality holds if and only if $z_{0}=0$, $x_{j_{0}} \in Z_{0}$ and
\begin{equation} \label{e242}
 \lim_{k \to \infty} \frac{\varepsilon _{\beta _{k}}}{\epsilon _{\beta _{k}} }=1 \ {\rm with}\ \epsilon _{\beta _{k}}>0\ {\rm given\ by}\ (\ref{e221}).
 \end{equation}
 Note from (\ref{e224}) and (\ref{e241}) that the equality in (\ref{e241}) holds, hence $z_{0}=0$, $x_{j_{0}} \in Z_{0}$ and (\ref{e242}) holds. Consequently, (\ref{e722}) and (\ref{e223}) follow from (\ref{e201}),(\ref{e202}) and (\ref{e242}). Moreover, using the fact that $z_{0}=0$, $x_{j_{0}} \in Z_{0}$, (\ref{e220}) follows from (\ref{e235}) and (\ref{e242}). This completes the proof of Theorem \ref{T3}.

 \section{Appendix}
In this appendix, we describe two basic results, which have played the key roles in the above analysis. For the completeness, the detailed proofs of the two results are presented below.
 \begin{lemma}\label{le1}Let $0<u_{1},u_{2}\in H^{\frac{1}{2} }(\mathbb{R} ^{3} )$ be two ground states of (\ref{e8}), then $\left \| u_{1} \right \|_{2}=\left \| u_{2} \right \|_{2}$, i.e., all positive ground states of (\ref{e8}) have the same $L^{2} $-norm.\\
 \end{lemma}
\begin{proof} One can obtain from Theorem 6.1 in \cite{Zhou} that the following  the Pohozaev identity holds for all solutions of (\ref{e8}).
   \begin{equation} \label{e10}
2\int\limits_{\mathbb{R} ^{3} }\left | (-\triangle )^{\frac{1}{4} }u  \right | ^{2}dx+3\int\limits_{\mathbb{R} ^{3} }u^{2}dx=\frac{5}{2}\iint\limits_{\mathbb{R} ^{3}\times \mathbb{R} ^{3}}\frac{u^{2}(x)u^{2}(y) }{\left | x-y \right | }dxdy.
	\end{equation}
Also, one can derive from (\ref{e8}) and (\ref{e10}) that any solution $u(x)$ of (\ref{e8}) satisfies
\begin{equation} \label{e11}
\int\limits_{\mathbb{R} ^{3} }\left | (-\triangle )^{\frac{1}{4} }u  \right | ^{2}dx=\int\limits_{\mathbb{R} ^{3} }u^{2}dx=\frac{1}{2}\iint\limits_{\mathbb{R} ^{3}\times \mathbb{R} ^{3}}\frac{u^{2}(x)u^{2}(y) }{\left | x-y \right | }dxdy.
	\end{equation}
The associated energy functional of (\ref{e8}) is given by
\begin{equation} \label{e12}
I(u)=\frac{1}{2} \int\limits_{\mathbb{R} ^{3} }\left | (-\triangle )^{\frac{1}{4} }u  \right | ^{2}+u^{2}dx-\frac{1}{4}\iint\limits_{\mathbb{R} ^{3}\times \mathbb{R} ^{3}}\frac{u^{2}(x)u^{2}(y) }{\left | x-y \right | }dxdy, \quad  u\in H^{\frac{1}{2} }(\mathbb{R} ^{3} ).
	\end{equation}
If $0<u\in H^{\frac{1}{2} }(\mathbb{R} ^{3} )$ satisfies $\left \langle I^{'}(u),\varphi   \right \rangle=0$  for any $\varphi \in H^{\frac{1}{2} }(\mathbb{R} ^{3} )$, then $u$ is called a positive weak solution of (\ref{e8}). Let
\begin{equation} \label{e13}
S:=\left \{u\in H^{\frac{1}{2} }(\mathbb{R} ^{3} ): \text{u \ is \ a\  positive\  weak\  solution\  of} \ (\ref{e8})  \right \},
	\end{equation}
\begin{equation} \label{e14}
G:=\left \{u\in S: I(u)\le I(v) \quad \forall v\in S  \right \}.
	\end{equation}
Then,  each element in $G$ is called a ground state of (\ref{e8}). Next, we claim if $u_{1}, u_{2}\in G$ and $u_{3}\in S$, then we have
\begin{equation} \label{e15}
\left \| u_{1} \right \| _{2}= \left \| u_{2} \right \| _{2}\le \left \| u_{3} \right \| _{2}.
	\end{equation}
In fact, for any $u\in S$, combining (\ref{e11}) and (\ref{e12}) yields that
\begin{equation} \label{e16}
I(u)=\frac{1}{2} \int\limits_{\mathbb{R} ^{3} }u^{2}dx.
	\end{equation}
Thus, the above claim follows by applying (\ref{e16}) and the definition of $G$ and therefore the proof is finished.
\end{proof}

 \begin{lemma}\label{le2} The refined Gagliardo-Nirenberg inequality (\ref{e27}) holds with the best constant $\frac{2}{a^{\ast } }$. Furthermore, each optimizer $\left ( u_{10}, u_{20} \right )$ of (\ref{e27}) has the form
 \begin{equation} \label{e18}
\left ( u_{10}(x), u_{20}(x) \right )=\left ( \tau sin\theta  Q\left (\alpha x+\eta  \right ) ,\tau cos\theta Q(\alpha x+\eta ) \right ),
	\end{equation}
where $\tau>0, \alpha>0, \theta \in \left [ 0,2\pi  \right )$ and $\eta \in \mathbb{R} ^{3}$.\\

 \end{lemma}
\begin{proof} Following the idea of \cite{Du22}, let us consider
$$ m:=\inf_{\left ( 0,0 \right )\ne (u_{1},u_{2} )\in H^{\frac{1}{2} }(\mathbb{R} ^{3})\times H^{\frac{1}{2} }(\mathbb{R} ^{3}) }J(u_{1},u_{2}),$$	
where
 $$
J\left ( u_{1},u_{2} \right ):= \frac{\displaystyle\int\limits_{\mathbb{R} ^{3} }\left (\left | (-\triangle )^{\frac{1}{4} }u_{1}  \right |^{2}+\left | (-\triangle )^{\frac{1}{4} }u_{2}  \right |^{2}\right )dx\int\limits_{\mathbb{R} ^{3} }\left (u_{1}^{2}+u_{2}^{2}\right )dx}{D\left ( u_{1}^{2}+u_{2}^{2},u_{1}^{2}+u_{2}^{2} \right ) }.
$$
Take $\left ( u_{1},u_{2} \right )=\left ( \frac{1}{\sqrt{2} }Q,\frac{1}{\sqrt{2} }Q  \right )$ as a test term, from (\ref{e21}), we thus infer
\begin{equation} \label{e22}
m\le J\left ( \frac{1}{\sqrt{2} }Q,\frac{1}{\sqrt{2} }Q  \right )=\frac{a^{\ast } }{2}.
\end{equation}
Moreover, from convexity inequality (see Theorem 7.13 in \cite{Galewski}) and (\ref{e17}), one can derive
\begin{equation} \label{e28}
J\left ( u_{1},u_{2} \right )\ge \frac{\displaystyle\int\limits_{\mathbb{R} ^{3} }\left | (-\triangle )^{\frac{1}{4} }\sqrt{u_{1}^{2}+u_{2}^{2}}    \right |^{2}dx\int\limits_{\mathbb{R} ^{3} }\left | \sqrt{u_{1}^{2}+u_{2}^{2}} \right | ^{2}  dx}{D\left ( u_{1}^{2}+u_{2}^{2},u_{1}^{2}+u_{2}^{2} \right ) }
\ge \frac{a^{\ast } }{2},
\end{equation}
where $\left ( u_{1},u_{2}  \right )\in H^{\frac{1}{2} }(\mathbb{R} ^{3})\times H^{\frac{1}{2} }(\mathbb{R} ^{3})\setminus \left \{ \left ( 0,0 \right )  \right \}$. Combining (\ref{e22}) and (\ref{e28}), we thus deduce
$$ m= J\left ( \frac{1}{\sqrt{2} }Q,\frac{1}{\sqrt{2} }Q  \right )=\frac{a^{\ast } }{2},$$
which shows that (\ref{e27}) holds with the best constant $\frac{ 2 }{a^{\ast}}$.
Note that, if $\left ( u_{10},u_{20} \right )$ is an optimizer of inequality (\ref{e27}), we then follow from Theorem 7.13 in \cite{Galewski} that $u_{20}=cu_{10}$ a.e. for some constant $c$. We therefore conclude that
$$\iint\limits_{\mathbb{R} ^{3}\times \mathbb{R} ^{3}}\frac{u_{10} ^{2}(x)u_{10}^{2}(y) }{\left | x-y \right | }dxdy = \frac{2}{a^{\ast } } \int\limits_{\mathbb{R} ^{3} }\left | (-\triangle )^{\frac{1}{4} }u_{10}  \right |^{2}dx\int\limits_{\mathbb{R} ^{3} }u_{10}^{2}dx,$$
which indicates that the classical Gagliardo-Nirenberg inequality (\ref{e17}) can be achieved at $u_{10}$. Similarly as in Lemma B.1 in \cite{Amrouss23}, we infer
$$u_{10}=\nu Q\left ( \alpha x+\eta  \right )$$
with some $\nu\ne 0,\alpha > 0$ and $ \eta \in \mathbb{R} ^{3}$. Therefore, one can deduce
$$  \left ( u_{10}(x),u_{20}(x) \right ) =\left ( \tau sin\theta  Q\left ( \alpha x+\eta  \right ),\tau cos\theta Q\left ( \alpha x+\eta  \right )   \right )$$
with some $\tau > 0,\alpha > 0,\theta \in \left [ 0,2\pi  \right )$ and $ \eta \in \mathbb{R} ^{3}$, which implies the proof is finished.

\end{proof}

\section*{Competing interests}
The authors declare that they have no conflicts of interest.

\section*{Authors contribution}
Each of the authors contributed to each part of this study equally, all authors read and approved the final manuscript.

\section*{Availability of supporting data}
Data sharing does not apply to this article as no data sets were generated or analyzed during the current study.

\section*{Funding}
Huiting He was supported by the Guangdong Basic and Applied Basic Research
Foundation (2024A1515012389) and the Guangzhou University Postgraduate Creative Capacity
Development Grant Scheme (Project No.2022GDJC-M01).

Chungen Liu was supported by the National Natural Science Foundation of China (Grant No.12171108) and the Guangdong Basic and Applied Basic Research
Foundation (2024A1515012389).

Jiabin Zuo was supported by 
 the Guangdong Basic and Applied Basic Research
Foundation (2024A1515012389).
\medskip
\bibliographystyle{plain}


\end{document}